\documentclass[11pt]{amsart}
\usepackage{amsmath,amssymb,amsfonts,url,mathptmx,xcolor}
\numberwithin{equation}{section}

\newtheorem{thm}{Theorem}[section]
\newtheorem{lem}{Lemma}[section]
\newtheorem{rem}{Remark}[section]
\newtheorem{prop}{Proposition}[section]

  \newcommand{\N}{\mathbb{N}}
  \def\Re{{\rm Re}\,}
\def\Im{{\rm Im}}

\newcommand{\R}{\mathbb{R}}
  \newcommand{\C}{\mathbb{C}}
  \newcommand{\beq}{\begin{equation}}
\newcommand{\eeq}{\end{equation}}
\newcommand{\into}{\int_{\Omega}}
\newcommand{\la}{\lambda}
\newcommand{\e}{\varepsilon}
\def\bbm[#1]{\mbox{\boldmath $#1$}}
\newcommand{\de}{\delta}
\newcommand{\al}{\alpha}

\begin{document}
\title[Singular Liouville Equation]{High Order Vanishing Theorems for Nonsimple Blowup Solutions of Singular Liouville Equation}

\author{Teresa D'Aprile}
\address{Dipartimento di Matematica, Universit\`a di Roma “Tor Vergata”, via della Ricerca
Scientifica 1, 00133 Roma . Italy.} \email{daprile@mat.uniroma2.it}

\author{Juncheng Wei}
\address{Department of Mathematics \\ Chinese University of Hong Kong\\ Shatin, NT, Hong Kong} \email{wei@math.cuhk.edu.hk }

\author{Lei Zhang} \footnote{The research of J. Wei is partially supported by Hong Kong General Research Fund "New frontiers in singularity analysis of nonlinear partial differential equations". Lei Zhang is partially supported by a Simons Foundation Collaboration Grant. T. D'Aprile acknowledges the MUR Excellence Department Project MatMod@TOV awarded to the Department of Mathematics, University of Rome Tor Vergata, CUP E83C23000330006}
\address{Department of Mathematics\\
        University of Florida\\
        1400 Stadium Rd\\
        Gainesville FL 32611}
\email{leizhang@ufl.edu}

\date{\today}

\begin{abstract} For a singular Liouville equation, it is plausible that a non-simple blowup phenomenon occurs around a quantized singular pole. The presence of complex blowup profiles of bubbling solutions presents substantial challenges in applications. In this article, we demonstrate that under natural assumptions, non-simple blowup takes place only when the derivatives of certain coefficient functions approach zero. Our main result encompasses all previous findings and determines the vanishing order for any specific quantized singular source. Our theorems can be utilized not only to eliminate multiple non-simple blowup scenarios in applications but also to investigate blowup solutions with moving poles.
\end{abstract}

\maketitle

\section{Introduction}
It is well known that certain partial differential equations serve as bridges that connect different fields of mathematics and physics. In particular a mean-field type equation defined on a Riemann surface $(M,g)$ of the form
$$\Delta_g u+2K_g(x)=2K(x)e^v-4\pi\sum_{j=1}^m\alpha_j\delta_{p_j}$$
connects conformal geometry and physics. A singular source $p_j$ is called quantized if the corresponding $\alpha_j$ is a positive integer. One essential difficulty is to study the profile of blow-up solutions if the blowup point happens to be a quantized singular source. When we focus on the locally defined equation, the purpose of the this article is to study the blowup solutions of 
\begin{equation}\label{main-2}
\Delta u+|x|^{2N}\mathrm{H}(x)e^{u}=0, 
\end{equation}
in a neighborhood of the origin in $\mathbb R^2$. Here, $\mathrm{H}$ is a positive smooth function and $N\in \mathbb N$ is a positive integer. Since the analysis is local in nature, we focus the discussion in a neighborhood of the origin:
Let $\mathfrak{u}_k$ be a sequence of solutions of
\begin{equation}\label{t-u-k}
\Delta \mathfrak{u}_k(x)+|x|^{2N}\mathrm{H}_k(x)e^{\mathfrak{u}_k}=0, \quad \mbox{in}\quad B_{\tau}
\end{equation}
for some $\tau>0$ independent of $k$. $B_{\tau}$ is the ball centered at the origin with radius $\tau$. $N$ in this article is a positive integer and we seek to study the profile of blow-up solutions if $0$ is the only blowup point. Now we state the usual assumptions on $\mathfrak{u}_k$ and $\mathrm{H}_k$:
For a positive constant $C$ independent of $k$, the following holds:
\begin{equation}\label{assumption-1}
\left\{\begin{array}{ll}
\|\mathrm{H}_k\|_{C^{N+2}(\bar B_{\tau})}\le C, \quad \frac 1C\le \mathrm{H}_k(x)\le C, \quad x\in \bar B_{\tau}, \\ \\
\int_{B_{\tau}} \mathrm{H}_k e^{\mathfrak{u}_k}\le C,\\  \\
|\mathfrak{u}_k(x)- \mathfrak{u}_k(y)|\le C, \quad \forall x,y\in \partial B_{\tau},
\end{array}
\right.
\end{equation}
and since we study the asymptotic behavior of blowup solutions around the singular source, we assume that there is no blowup point except at the origin:
\begin{equation}\label{assump-2}
\max_{K\subset\subset B_{\tau}\setminus \{0\}} \mathfrak{u}_k\le C(K).
\end{equation}

 If a sequence of solutions $\{u^k\}_{k=1}^{\infty}$ of (\ref{main-2}) satisfies
 $$\lim_{k\to \infty}u^k(x_k)=\infty,\quad \mbox{ for some $\bar x\in B_{\tau}$ and $x_k\to \bar x$,} $$
  we say $\{u^k\}$ is a sequence of bubbling solutions or blowup solutions, $\bar x$ is called a blowup point. The question we consider in this work is, when $0$ is the only blow-up point in a neighborhood of the origin, what vanishing theorems will the coefficient functions $\mathrm{H}_k$ satisfy?

One indispensable assumption is that the blowup solutions violate the spherical Harnack inequality around the origin:
\begin{equation}\label{no-sp-h}
\max_{x\in B_{\tau}} \mathfrak{u}_k(x)+2(1+N)\log |x|\to \infty,
\end{equation}
 It is also mentioned in literature ( see \cite{kuo-lin-jdg, wei-zhang-adv} ) that $0$ is called an non-simple blowup point. The authors have made progress in the study of non-simple blow-up solutions. In particular, Wei and Zhang proved in \cite{wei-zhang-jems} the following Laplacian vanishing theorem:

\emph{Theorem A: (Wei-Zhang):
Let $\{\mathfrak{u}_k\}$ be a sequence of solutions of (\ref{t-u-k}) such that (\ref{assumption-1}),(\ref{assump-2}) hold and the spherical Harnack inequality is violated as in (\ref{no-sp-h}). Then along a sub-sequence
$$\lim_{k\to \infty}\Delta (\log \mathrm{H}_k)(0)=0. $$
}
and in \cite{wei-zhang-plms} the first-order vanishing theorem :

\emph{Theorem B (Wei-Zhang, \cite{wei-zhang-plms}:  Let $\{\mathfrak{u}_k\}$ be a sequence of solutions of (\ref{t-u-k}) such that (\ref{assumption-1}),(\ref{assump-2}) and (\ref{no-sp-h}) hold. Then along a subsequence
$$\lim_{k\to \infty}\nabla (\log\mathrm{H}_k+\phi_k)(0)=0$$ where $\phi_k$ is defined as
\begin{equation}\label{phi-k}
\left\{\begin{array}{ll}
\Delta \phi_k(x)=0,\quad \mbox{in}\quad B_{\tau},\\ \\
\phi_k(x)=\mathfrak{u}_k(x)-\frac{1}{2\pi \tau}\int_{\partial B_{\tau}}\mathfrak{u}_kdS,\quad x\in \partial B_{\tau}.
\end{array}
\right.
\end{equation}
}
The equation (\ref{main-2}) comes from its equivalent form 
$$\Delta v+\mathrm{H}e^v=4\pi N\delta_0 $$
by using a logarithmic function to eliminate the Dirac mass on the right-hand side. The study of blowup solutions for (\ref{main-2}) gives precise description of bubbling profile for global equations such as 
the following mean field equation defined on a Riemann surface $(M,g)$:
\begin{equation}\label{mean-sin}
\Delta_gu+\rho(\frac{h(x)e^{u(x)}}{\int_Mhe^{u}}-1)=4\pi\sum_{t=1}^M \alpha_t (\delta_{p_t}-1).
\end{equation}
Equation (\ref{mean-sin}) describes a conformal metric with prescribed conic singularities (see \cite{erem-3,tr-1,tr-2}). In this context, $h$ is a positive smooth function, $\rho>0$ is a constant, and the volume of $M$ is assumed to be $1$ for convenience. Additionally, $\alpha_j > -1$ are constants. When the singular source is quantized ($\alpha \in \mathbb{N}$), the equation is deeply connected to Algebraic geometry, integrable systems, number theory, and complex Monge-Ampère equations (see \cite{chen-lin-last-cpam}). In physics, this main equation reveals critical features of the mean field limits of point vortices in Euler flow \cite{caglioti-1,caglioti-2}, models in the Chern-Simons-Higgs theory \cite{jackiw}, and the electroweak theory \cite{ambjorn}, among others.

 The non-simple bubbling situation has been observed in the Liouville equation \cite{kuo-lin-jdg,bart3}, Liouville systems \cite{gu-zhang-1,gu-zhang-2, wu-zhang-siam}, and fourth-order equations \cite{ahmedou-wu-zhang}. Although it has been established in \cite{wei-zhang-plms,wei-zhang-jems} that the first derivatives and the Laplace of the coefficient function vanish at the quantized singular source, it is still highly desirable to prove even higher-order vanishing theorems. The elimination of nonsimple blow-up situations significantly simplifies the profiles of bubbling solutions and is crucial for various subsequent applications. This article aims to demonstrate the vanishing of coefficients for any high order. Our first main result is to prove the vanishing of all second derives of the coefficient function for any $N\ge 1$.

\begin{thm}\label{main-thm-1} Let $\{\mathfrak{u}_k\}$ be a sequence of solutions of (\ref{t-u-k}) such that (\ref{assumption-1}),(\ref{assump-2}) and  (\ref{no-sp-h}) hold. Then along a sub-sequence
$$|D^2(\log \mathrm{H}_k+\phi_k)(0)|=o(1). $$
\end{thm}
We also have an extension of Theorem \ref{main-thm-1}, which is concerned about vanishing estimates for higher order derivatives. 
Our second main result is 

\begin{thm}\label{main-thm-2} Let $\{\mathfrak{u}_k\}$ be a sequence of solutions of (\ref{t-u-k}) such that (\ref{assumption-1}),(\ref{assump-2}) and  (\ref{no-sp-h}) hold. Then for $N\ge 2^{M+1}$ where $M\in \mathbb N\cup\{0\}$ is a nonnegative integer, we have 
\[|D^{\alpha}(\log \mathrm{H}_k+\phi_k)(0)|=o(1),\quad \forall |\alpha |\le 2^M+1, \]
along a subsequence. 
\end{thm}

\begin{rem}
Theorem \ref{main-thm-1} holds for all $N\ge 1$. Theorem \ref{main-thm-2} proves that any order of derivatives of the coefficient function vanishes as long as $N$ is large.   Both theorems improve the previous Laplacian vanishing Theorem in \cite{wei-zhang-jems}. For some applications such as the study of the blow-up profile of singular Liouville equations with moving poles, the Laplacian vanishing theorem in \cite{wei-zhang-jems} cannot be applied. It is crucial to have Theorem \ref{main-thm-1} and Theorem \ref{main-thm-2}. 
\end{rem}

{\bf Notation:} We will use $B(x_0,r)$ to denote a ball centered at $x_0$ with radius $r$. If $x_0$ is the origin, we use $B_r$. $C$ represents a positive constant that can change from place to place.

The proof of the main theorem requires delicate point-wise estimates of the blow-up solutions around each local maximum point. This was the case in the proofs of the previous work of the second and third authors in \cite{wei-zhang-adv,wei-zhang-plms,wei-zhang-jems}. The difference in this work is that we need more refined estimates, which require a subtler analysis. In general, the nature of the proof demands estimates of different levels of precision in a progressive manner. One level of estimate leads to a higher-order estimate later. Identifying the error threshold in each stage is crucial for the completion of the whole proof eventually. If we describe the entire scheme of the proof from a general viewpoint, we use Fourier analysis for the pointwise estimate around each local maximum, Harnack inequality to pass certain smallness information to regions away from local maximums. Moreover we use Pohozaev identities to capture crucial vanishing information on coefficient functions. The approach we employ in this work should be useful for other related situations as well. 

The organization of this article is the following. Sections 2-5 contain the complete proof of Theorem \ref{main-thm-1}. Section six is saved for the proof of Theorem \ref{main-thm-2}. Finally in section seven we provide a different perspective by discussing the Dirichlet problem and giving a second proof of the Hessian vanishing estimate for $N=1$ using new Pohozaev identities.

\section{Preliminary discussions for blowup analysis and the locally defined equation}
In the first stage of the proof of the main Theorem \ref{main-thm-1} we set up some notations and cite some preliminary results.
Let
\begin{equation}\label{uk-d}
u_k(x)=\mathfrak{u}_k(x)-\phi_k(x), \quad \mbox{where $\phi_k$ is defined in (\ref{phi-k}) and}
\end{equation}
\begin{equation}\label{hk-d}
h_k(x)=\mathrm{H}_k(x)e^{\phi_k(x)}.
\end{equation}
Now we write the equation of $u_k$ as
\begin{equation}\label{eq-uk}
\Delta u_k(x)+|x|^{2N}h_k(x)e^{u_k}=0,\quad \mbox{ in }\quad B_{\tau}
\end{equation}
Without loss of generality we assume
\begin{equation}\label{rea-h}
\lim_{k\to \infty} h_k(0)=1.
\end{equation}

Obviously (\ref{no-sp-h}) is equivalent to
\begin{equation}\label{no-sp-h-u}
\max_{x\in B_{\tau}} u_k(x)+2(1+N)\log |x|\to \infty,
\end{equation}

It is well known \cite{kuo-lin-jdg, bart3} that $ u_k$ exhibits a non-simple blow-up profile.  It is established in \cite{kuo-lin-jdg,bart3} that there are $N+1$ local maximum points of $ u_k$: $p_0^k$,....,$p_{N}^k$ and they are evenly distributed on $\mathbb S^1$ after scaling according to their magnitude: Suppose along a subsequence
$$\lim_{k\to \infty}p_0^k/|p_0^k|=e^{i\theta_0}, $$
then
$$\lim_{k\to \infty} \frac{p_l^k}{|p_0^k|}=e^{i(\theta_0+\frac{2\pi l}{N+1})}, \quad l=1,...,N. $$
For many reasons it is convenient to denote $|p_0^k|$ as $\delta_k$ and define $\mu_k$ as follows:
\begin{equation}\label{muk-dk}
\delta_k=|p_0^k|\quad \mbox{and }\quad \mu_k= u_k(p_0^k)+2(1+N)\log \delta_k.
\end{equation}
Also we use 
\begin{equation}\label{ep-k}
\epsilon_k=e^{-\frac 12 \mu_k} 
\end{equation}
to be the scaling factor most of the time. 
Since $p_l^k$'s are evenly distributed
around $\partial B_{\delta_k}$, standard results for Liouville equations around a regular blow-up point can be applied to have $ u_k(p_l^k)= u_k(p_0^k)+o(1)$. Also, (\ref{no-sp-h}) gives $\mu_k\to \infty$. The interested readers may look into \cite{kuo-lin-jdg,bart3} for more detailed information.

\section{Approximating bubbling solutions by global solutions}

We write $p_0^k$ as $p_0^k=\delta_ke^{i\theta_k}$ and define $v_k$ as
\begin{equation}\label{v-k-d}
v_k(y)=u_k(\delta_k ye^{i\theta_k})+2(N+1)\log \delta_k,\quad |y|<\tau \delta_k^{-1}.
\end{equation}
If we write out each component, (\ref{v-k-d}) is
$$
v_k(y_1,y_2)=u_k(\delta_k(y_1\cos\theta_k-y_2\sin\theta_k),\delta_k(y_1\sin\theta_k+y_2\cos\theta_k))+2(1+N)\log \delta_k. $$
Then it is standard to verify that $v_k$ solves

\begin{equation}\label{e-f-vk}
\Delta v_k(y)+|y|^{2N}\mathfrak{h}_k(\delta_k y)e^{v_k(y)}=0,\quad |y|<\tau/\delta_k,
\end{equation}
where
\begin{equation}\label{frak-h}
\mathfrak{h}_k(x)=h_k(xe^{i\theta_k}),\quad |x|<\tau.
\end{equation}
Thus the image of $p_0^k$ after scaling is $Q_1^k=e_1=(1,0)$.
Let $Q_1^k$, $Q_2^k$,...,$Q_{N}^k$ be the images of $p_i^k$ $(i=1,...,N)$ after the scaling:
$$Q_l^k=\frac{p_l^k}{\delta_k}e^{-i\theta_k},\quad l=1,...,N. $$
 It is established by Kuo-Lin in \cite{kuo-lin-jdg} and independently by Bartolucci-Tarantello in \cite{bart3} that
\begin{equation}\label{limit-q}
\lim_{k\to \infty} Q_l^k=\lim_{k\to \infty}p_l^k/\delta_k=e^{\frac{2l\pi i}{N+1}},\quad l=0,....,N.
\end{equation}
Then it is proved in \cite{wei-zhang-adv} that ( see (3.13) in \cite{wei-zhang-adv})
\begin{equation}\label{distance}
Q_l^k-e^{\frac{2\pi l i}{N+1}}=O(\mu_ke^{-\mu_k})+O(|\nabla \log \mathfrak{h}_k(0)|\delta_k). 
\end{equation}
Using the rate of $\nabla \mathfrak{h}_k(0)$ in \cite{wei-zhang-adv} we have
\begin{equation}\label{Qm-close}
Q_l^k-e^{\frac{2\pi l i}{N+1}}=O(\mu_ke^{-\mu_k})+O(\delta_k^2).
\end{equation}
Choosing $\epsilon>0$ small and independent of $k$, we can make disks centered at $Q_l^k$ with radius $3\epsilon$ (denoted as $B(Q_l^k,3\epsilon ) $) mutually disjoint. Let
\begin{equation}\label{v-muk}
\mu_k=\max_{B(Q_0^k,\epsilon)} v_k.
\end{equation}
Since $Q_l^k$ are evenly distributed around $\partial B_1$, it is easy to use standard estimates for single Liouville equations (\cite{zhangcmp,gluck,chenlin1}) to obtain
$$\max_{B(Q_l^k,\epsilon)}v_k=\mu_k+o(1),\quad l=1,...,N. $$

Let
\begin{equation}\label{def-Vk}
V_k(x)=\log \frac{e^{\mu_k}}{(1+\frac{e^{\mu_k}\mathfrak{h}_k(\delta_k e_1)}{8(1+N)^2}|y^{N+1}-e_1|^2)^2}.
\end{equation}
Clearly $V_k$ is a solution of
\begin{equation}\label{eq-for-Vk}
\Delta V_k+\mathfrak{h}_k(\delta_k e_1)|y|^{2N}e^{V_k}=0,\quad \mbox{in}\quad \mathbb R^2, \quad V_k(e_1)=\mu_k.
\end{equation}
This expression is based on the classification theorem of Prajapat-Tarantello \cite{prajapat}.
For convenience we use
$$\beta_l=\frac{2\pi l}{N+1}, \quad \mbox{so}\,\, e_1=e^{i\beta_0}=Q_0^k,\quad
e^{i\beta_l}=Q_l^k+E,\,\,\mbox{ for }\,\, l=1,...,N. $$
The following expansion of $V_k$ on $|y|=\tau\delta_k^{-1}$ shows that the oscillation of $V_k$ is $O(\delta_k^{N+1})$ on $\partial B(0,\tau\delta_k^{-1})$:
\begin{equation}\label{osi-global}
V_k(x)=-\mu_k+2\log\frac{8(1+N)^2}{\mathfrak{h}_k(\delta_ke_1)}-4\log L_k+4L_k^{-N-1}\cos ((N+1)\theta)+O(\delta_k^{2N+2})
\end{equation}
where $L_k=\tau\delta_k^{-1}$.
\section{Vanishing of the first derivatives}
Our first goal is to prove the following vanishing rate for $\nabla \mathfrak{h}_k(0)$:
\begin{thm}\label{vanish-first-h}
\begin{equation}\label{vanish-first-tau}
\nabla (\log \mathfrak{h}_k)(0)=O(\delta_k)
\end{equation}
\end{thm}

\noindent{\bf Proof of Theorem \ref{vanish-first-h}:}

Note that we have proved in \cite{wei-zhang-adv} that 
$$\nabla (\log \mathfrak{h}_k)(0)=O(\delta_k^{-1}\mu_ke^{-\mu_k})+O(\delta_k).$$
From (\ref{ep-k}) we see that if $\delta_k\ge C\epsilon_k\mu_k^{\frac 12}$, there is nothing to prove. So we assume that 
\begin{equation}\label{delta-small-1}
\delta_k=o(\epsilon_k\mu_k^{\frac 12}).
\end{equation}
By way of contradiction we assume that there exists $C>0$ independent of $k$, such that 
\begin{equation}\label{assum-delta-h}
|\nabla \mathfrak{h}_k(0)|/\delta_k \to \infty. 
\end{equation}

Another observation is that based on (\ref{Qm-close}) we have 
\begin{equation}\label{location-Q}
\epsilon_k^{-1} | Q_l^k-e^{i\beta_l}|\to 0,\quad l=0,...,N.
\end{equation}
Thus $\xi_k$ tends to $U$ after scaling. 

Under assumption (\ref{delta-small-1}) we cite Proposition 3.1 of \cite{wei-zhang-plms}: 

\emph{Proposition 3.1 of \cite{wei-zhang-plms}: Let $l=0,...,N$ and $\tau_1$ be small so that $B(e^{i\beta_l},\delta)\cap B(e^{i\beta_s},\delta)=\emptyset$ for $l\neq s$.
In each $B(e^{i\beta_l},\delta)$
\begin{equation}\label{global-close}
|v_k(x)-V_k(x)|\le \left\{\begin{array}{ll}
C\mu_ke^{-\mu_k/2},\quad |x-e^{i\beta_l}|\le Ce^{-\mu_k/2}, \\
\\
C\frac{\mu_ke^{-\mu_k}}{|x-e^{i\beta_l}|}+O(\mu_k^2e^{-\mu_k}),\quad Ce^{-\mu_k/2}\le |x-e^{i\beta_l}|\le \tau_1.
\end{array}
\right.
\end{equation} }


One major step in the proof of Theorem \ref{vanish-first-h} is the following estimate:

\begin{prop}\label{key-w8-8} Let $w_k=v_k-V_k$, then
$$|w_k(y)|\le C\tilde{\delta_k}, \quad y\in \Omega_k:=B(0,\tau \delta_k^{-1}), $$
where $\tilde{\delta_k}=|\nabla \mathfrak{h}_k(0)|\delta_k+\delta_k^2$.
\end{prop}

\noindent{\bf Proof of Proposition \ref{key-w8-8}:}

Obviously based on (\ref{assum-delta-h}) we have $|\nabla \mathfrak{h}_k(0)|\delta_k/\delta_k^2\to \infty$. 
Now we recall the equation for $v_k$ is (\ref{e-f-vk}),
$v_k$ is a constant on $\partial B(0,\tau \delta_k^{-1})$. Moreover $v_k(e_1)=\mu_k$. Recall that $V_k$ defined in (\ref{def-Vk}) satisfies
$$\Delta V_k+\mathfrak{h}_k(\delta_ke_1)|y|^{2N}e^{V_k}=0,\quad \mbox{in}\quad \mathbb R^2, \quad \int_{\mathbb R^2}|y|^{2N}e^{V_k}<\infty, $$
$V_k$ has its local maximums at $e^{i\beta_l}$ for $l=0,...,N$ and $V_k(e_1)=\mu_k$.
For $|y|\sim \delta_k^{-1}$, the oscillation of $V_k$ is given by (\ref{osi-global}).

Let $\Omega_k=B(0,\tau_1 \delta_k^{-1})$,  we shall derive a precise, point-wise estimate of $w_k$ in $B_3\setminus \cup_{l=1}^{N}B(Q_l^k,\tau_1)$ where $\tau_1>0$ is a small number independent of $k$. Here we note that among $N+1$ local maximum points, we already have $e_1$ as a common local maximum point for both $v_k$ and $V_k$ and we shall prove that $w_k$ is very small in $B_3$ if we exclude all bubbling disks except the one around $e_1$. Before we carry out more specific computation we emphasize the importance of
\begin{equation}\label{control-e}
w_k(e_1)=|\nabla w_k(e_1)|=0.
\end{equation}
Now we write the equation of $w_k$ as
\begin{equation}\label{eq-wk}
\Delta w_k+\mathfrak{h}_k(\delta_k y)|y|^{2N}e^{\xi_k}w_k=(\mathfrak{h}_k(\delta_k e_1)-\mathfrak{h}_k(\delta_k y))|y|^{2N}e^{V_k}
\end{equation}
 in $\Omega_k$, where $\xi_k$ is obtained from the mean value theorem:
$$
e^{\xi_k(x)}=\left\{\begin{array}{ll}
\frac{e^{v_k(x)}-e^{V_k(x)}}{v_k(x)-V_k(x)},\quad \mbox{if}\quad v_k(x)\neq V_k(x),\\
\\
e^{V_k(x)},\quad \mbox{if}\quad v_k(x)=V_k(x).
\end{array}
\right.
$$
An equivalent form is
\begin{equation}\label{xi-k}
e^{\xi_k(x)}=\int_0^1\frac d{dt}e^{tv_k(x)+(1-t)V_k(x)}dt=e^{V_k(x)}\big (1+\frac 12w_k(x)+O(w_k(x)^2)\big ).
\end{equation}
For convenience we write the equation for $w_k$ as
\begin{equation}\label{eq-wk-2}
\Delta w_k+\mathfrak{h}_k(\delta_k y)|y|^{2N}e^{\xi_k}w_k=\delta_k\nabla \mathfrak{h}_k(\delta_k e_1)\cdot (e_1-y)|y|^{2N}e^{V_k}+E_1
\end{equation}
where $$E_1=O(\delta_k^2)|y-e_1|^2|y|^{2N}e^{V_k},\quad y\in \Omega_k. $$
Note that the oscillation of $w_k$ on $\partial \Omega_k$ is $O(\delta_k^{N+1})$, which all comes from the oscillation of $V_k$. 

Let $M_k=\max_{x\in \bar \Omega_k}|w_k(x)|$. We shall get a contradiction by assuming $M_k/\tilde{\delta_k}\to \infty$. This assumption implies
\begin{equation}\label{big-mk}
M_k/\delta_k^2\to \infty. 
\end{equation}
Set
$$\tilde w_k(y)=w_k(y)/M_k,\quad x\in \Omega_k. $$
Clearly $\max_{x\in \Omega_k}|\tilde w_k(x)|=1$. The equation for $\tilde w_k$ is
\begin{equation}\label{t-wk}
\Delta \tilde w_k(y)+|y|^{2N}\mathfrak{h}_k(\delta_k e_1)e^{\xi_k}\tilde w_k(y)=a_k\cdot (e_1-y)|y|^{2N}e^{V_k}+\tilde E_1,
\end{equation}
in $\Omega_k$,
where $a_k=\delta_k\nabla \mathfrak{h}_k(0)/M_k\to 0$,
\begin{equation}\label{t-ek}
\tilde E_1=o(1)|y-e_1|^2|y|^{2N}e^{V_k},\quad y\in \Omega_k.
\end{equation}
Also on the boundary, since $M_k/\delta_k^2\to \infty$, we have 
\begin{equation}\label{w-bry}
\tilde w_k=C+o(1),\quad \mbox{on}\quad \partial \Omega_k. 
\end{equation}
Let
\begin{equation}\label{w-ar-e1}
W_k(z)=\tilde w_k(e_1+\epsilon_kz)
\end{equation}
and we recall that $\epsilon_k=e^{-\frac 12 \mu_k}$.
then if we use $W$ to denote the limit of $W_k$, we have
$$\Delta W+e^UW=0, \quad \mathbb R^2, \quad |W|\le 1, $$
and $U$ is a solution of $\Delta U+e^U=0$ in $\mathbb R^2$ with $\int_{\mathbb R^2}e^U<\infty$. Since $0$ is the local maximum of $U$,
we know from the classification theorem of Caffarelli-Gidas-Spruck \cite{CGS} that 
$$U(z)=\log \frac{1}{(1+\frac 18|z|^2)^2}. $$
If we use $\bar\xi_{k,0}(|z|)$ to be the radial part of 
\[\log (|e_1+\epsilon_k z|^{2N}\mathfrak{h}_k(\delta_ke_1))+\xi_k(e_1+\epsilon_kz)+2\log \epsilon_k\] 
with respect to $e_1$, which satisfies $\bar \xi_{k,0}\to U$ in $C^2_{loc}(\mathbb R^2)$ and 
\[e^{\bar \xi_{k,0}(z)}\le C(1+|z|)^{-4},\quad |z|\le \tau_k \epsilon_k^{-1}. \]
We write the equation of $W_k$ as
\begin{equation}\label{e-Wk}
\Delta W_k(z)+e^{\bar \xi_{k,0}}W_k=E_2^k,\quad |z|\le \tau_1\epsilon_k^{-1}
\end{equation}
where 
\[E_2^k(z)=O(\epsilon_k)(1+|z|)^{-3}.\]
   In the following we shall put the proof of Proposition \ref{key-w8-8} into a few estimates. In the first estimate we prove

\begin{lem}\label{w-around-e1} For $\delta>0$ small and independent of $k$,
\begin{equation}\label{key-step-1}
\tilde w_k(y)=o(1),\quad \nabla \tilde w_k=o(1) \quad \mbox{in}\quad B(e_1,\delta)\setminus B(e_1,\delta/8)
\end{equation}
where $B(e_1,3\delta)$ does not include other blowup points.
\end{lem}

\noindent{\bf Proof of Lemma  \ref{w-around-e1}:}

If (\ref{key-step-1}) is not true, we have, without loss of generality that
\begin{equation}\label{t-w-c}
\tilde w_k\to c>0.
\end{equation}
This is based on the fact that $\tilde w_k$ tends to a global harmonic function with removable singularity. So $\tilde w_k$ tends to constant. Here we assume $c>0$ but the argument for $c<0$ is the same. Recall that $W_k$ is defined in (\ref{w-ar-e1}).
Here we further claim that $W\equiv 0$ in $\mathbb R^2$ because $W(0)=|\nabla W(0)|=0$, a fact well known based on the classification of the kernel of the linearized operator. Going back to $W_k$, we have
$$W_k(z)=o(1),\quad |z|\le R_k \mbox{ for some } \quad R_k\to \infty. $$

Let
\begin{equation}\label{for-g0}
g_0^k(r)=\frac 1{2\pi}\int_0^{2\pi}W_k(r,\theta)d\theta.
\end{equation}
Then clearly $g_0^k(r)\to c>0$ for $r\sim \epsilon_k^{-1}$.
 The equation for $g_0^k$ is
\begin{align*}
&\frac{d^2}{dr^2}g_0^k(r)+\frac 1r \frac{d}{dr}g_0^k(r)+\mathfrak{h}_k(\delta_ke_1)e^{\bar \xi_{k,0}}g_0^k(r)=\tilde E_0^k(r)\\
&g_0^k(0)=\frac{d}{dr}g_0^k(0)=0.
\end{align*}
where
$$|\tilde E_0^k(r)|\le O(\epsilon_k)(1+r)^{-3}. $$

For the homogeneous equation, the two fundamental solutions are known: $g_{01}^k$, $g_{02}^k$, where, by elementary analysis, we obtain that $g_{01}^k$ tends to 
$$\frac{1-\frac 18r^2}{1+\frac 18 r^2}.$$
By the standard reduction of order process, $g_{02}^k(r)=O(\log r)$ for $r>1$ with bounds independent of $k$.
Then it is easy to obtain, assuming $|W_k(z)|\le 1$, that
\begin{align*}
|g_0^k(r)|\le C|g_{01}^k(r)|\int_0^r s|\tilde E_0^k(s) g_{02}^k(s)|ds+C|g_{02}^k(r)|\int_0^r s|g_{01}^k(s)\tilde E_0^k(s)|ds\\
\le C\epsilon_k\log (2+r). \quad 0<r<\delta_0 \epsilon_k^{-1}.
\end{align*}
Clearly this is a contradiction to (\ref{t-w-c}). We have proved $c=0$, which means $\tilde w_k=o(1)$ in $B(e_1, \delta_0)\setminus B(e_1, \delta_0/8)$.
Then it is easy to use the equation for $\tilde w_k$ and standard Harnack inequality to prove
$\nabla \tilde w_k=o(1)$ in the same region.
Lemma \ref{w-around-e1} is established. $\Box$

\medskip

The second estimate is a more precise description of $\tilde w_k$ around $e_1$:
\begin{lem}\label{t-w-1-better} For any given $\sigma\in (0,1)$ there exists $C>0$ such that
\begin{equation}\label{for-lambda-k}
|\tilde w_k(e_1+\epsilon_kz)|\le C\epsilon_k^{\sigma} (1+|z|)^{\sigma},\quad 0<|z|<\tau \epsilon_k^{-1}.
\end{equation}
for some $\tau>0$.
\end{lem}

\begin{rem}
    Lemma \ref{t-w-1-better} is an intermediate estimate for $\tilde w_k$. We eventually need to have an estimate starting with $o(\epsilon_k)$. The reason that we cannot obtian more precise estimate is because we only know $\tilde w=o(1)$ on $\partial B(e_1,\tau)$. More precise information is needed to have a better estimate.
\end{rem}
\noindent{\bf Proof of Lemma \ref{t-w-1-better}:} Let $W_k$ be defined as in (\ref{w-ar-e1}). In order to obtain a better estimate we need to write the equation of $W_k$ more precisely than (\ref{e-Wk}):
\begin{equation}\label{w-more}
\Delta W_k+\mathfrak{h}_k(\delta_ke_1)e^{\Theta_k}W_k=E_3^k(z), \quad z\in \Omega_{Wk}
\end{equation}
where
$\Theta_k$ is defined by
$$e^{\Theta_k(z)}=|e_1+\epsilon_k z|^{2N}e^{\xi_k(e_1+\epsilon_kz)+2\log \epsilon_k}, $$
$\Omega_{Wk}=B(0,\tau \epsilon_k^{-1})$ and $E_3^k(z)$ satisfies
$$E_3^k(z)=O(\epsilon_k)(1+|z|)^{-3},\quad z\in \Omega_{Wk}. $$
Here we observe that by Lemma \ref{w-around-e1}  $W_k=o(1)$ on $\partial \Omega_{Wk}$. 
Let
$$\Lambda_k=\max_{z\in \Omega_{Wk}}\frac{|W_k(z)|}{\epsilon_k^{\sigma}(1+|z|)^{\sigma}}. $$
If (\ref{for-lambda-k}) does not hold, $\Lambda_k\to \infty$ and we use $z_k$ to denote where $\Lambda_k$ is attained. Note that because of the smallness of $W_k$ on $\partial \Omega_{Wk}$, $z_k$ is an interior point. Let
$$g_k(z)=\frac{W_k(z)}{\Lambda_k (1+|z_k|)^{\sigma}\epsilon_k^{\sigma}},\quad z\in \Omega_{Wk}, $$
we see immediately that
\begin{equation}\label{g-sub-linear}
|g_k(z)|=\frac{|W_k(z)|}{\epsilon_k^{\sigma}\Lambda_k(1+|z|)^{\sigma}}\cdot \frac{(1+|z|)^{\sigma}}{(1+|z_k|)^{\sigma}}\le  \frac{(1+|z|)^{\sigma}}{(1+|z_k|)^{\sigma}}.
\end{equation}
Note that $\sigma$ can be as close to $1$ as needed. The equation of $g_k$ is
$$\Delta g_k(z)+\mathfrak{h}_k(\delta_k e_1)e^{\Theta_k}g_k=o(\epsilon_k^{1-\sigma})\frac{(1+|z|)^{-3}}{(1+|z_k|)^{\sigma}}, \quad \mbox{in}\quad \Omega_{Wk}. $$
Then we can obtain a contradiction to $|g_k(z_k)|=1$ as follows: If $\lim_{k\to \infty}z_k=P\in \mathbb R^2$, this is not possible because that fact that $g_k(0)=|\nabla g_k(0)|=0$ and the sub-linear growth of $g_k$ in (\ref{g-sub-linear}) implies that $g_k\to 0$ over any compact subset of $\mathbb R^2$ (see \cite{chenlin1,zhangcmp}). So we have $|z_k|\to \infty$. But this would lead to a contradiction again by using the Green's representation of $g_k$:
\begin{align} \label{temp-1}
&\pm 1=g_k(z_k)=g_k(z_k)-g_k(0)\\
&=\int_{\Omega_{k,1}}(G_k(z_k,\eta)-G_k(0,\eta))(\mathfrak{h}_k(\delta_k e_1)e^{\Theta_k}g_k(\eta)+o(\epsilon_k^{1-\sigma})\frac{(1+|\eta |)^{-3}}{(1+|z_k|)^{\sigma}})d\eta+o(1).\nonumber
\end{align}
where $G_k(y,\eta)$ is the Green's function on $\Omega_{Wk}$ and $o(1)$ in the equation above comes from the smallness of $W_k$ on $\partial \Omega_{Wk}$. Let $L_k=\tau\epsilon_k^{-1}$, the expression of $G_k$ is 
$$G_k(y,\eta)=-\frac{1}{2\pi}\log |y-\eta|+\frac 1{2\pi}\log (\frac{|\eta |}{L_k}|\frac{L_k^2\eta}{|\eta |^2}-y|). $$
$$G_k(z_k,\eta)-G_k(0,\eta)=-\frac{1}{2\pi}\log |z_k-\eta |+\frac 1{2\pi}\log |\frac{z_k}{|z_k|}-\frac{\eta z_k}{L_k^2}|+\frac 1{2\pi}\log |\eta |. $$
Using this expression in (\ref{temp-1}) we obtain from elementary computation that the right hand side of (\ref{temp-1}) is $o(1)$, a contradiction to $|g_k(z_k)|=1$. Lemma \ref{t-w-1-better} is
established. $\Box$

\medskip

The smallness of $\tilde w_k$ around $e_1$ can be used to obtain the following third key estimate:
\begin{lem}\label{small-other}
\begin{equation}\label{key-step-2}
\tilde w_k=o(1)\quad \mbox{in}\quad B(e^{i\beta_l},\tau)\quad l=1,..,N.
\end{equation}
\end{lem}

\noindent{\bf Proof of Lemma \ref{small-other}:}
We abuse the notation $W_k$ by defining it as
$$W_k(z)=\tilde w_k(e^{i\beta_l}+\epsilon_k z),\quad z\in \Omega_{k,l}:=B(0,\tau \epsilon_k^{-1}). $$
Here we point out that based on (\ref{Qm-close}) and (\ref{delta-small-1}) we have $\epsilon_k^{-1}|Q_l^k-e^{i\beta_l}|\to 0$. So the scaling around $e^{i\beta_l}$ or $Q_l^k$ does not affect the
limit function.

$$\epsilon_k^2 |e^{i\beta_l}+\epsilon_kz|^{2N}\mathfrak{h}_k(\delta_ke_1)e^{\xi_k(e^{i\beta_l}+\epsilon_kz)}\to e^{U(z)} $$
where $U(z)$ is a solution of
$$\Delta U+e^U=0,\quad \mbox{in}\quad \mathbb R^2, \quad \int_{\mathbb R^2}e^U<\infty. $$
Here we recall that $\lim_{k\to \infty} \mathfrak{h}_k(\delta_k e_1)=1$.
Since $W_k$ converges to a solution of the linearized equation:
$$\Delta W+e^UW=0, \quad \mbox{in}\quad \mathbb R^2. $$
If the growth of $W$ at infinity is sub-linear: $|W(x)|=o(1)|x|$ when $|x|\to \infty$, $W$ can be written as a linear combination of three functions:
$$W(x)=c_0\phi_0+c_1\phi_1+c_2\phi_2, $$
where
$$\phi_0=\frac{1-\frac 18 |x|^2}{1+\frac 18 |x|^2} $$
$$\phi_1=\frac{x_1}{1+\frac 18 |x|^2},\quad \phi_2=\frac{x_2}{1+\frac 18|x|^2}. $$
The remaining part of the proof consisting of proving $c_0=0$ and $c_1=c_2=0$. First we prove $c_0=0$.

\noindent{\bf Step one: $c_0=0$.}
First we write the equation for $W_k$ in a convenient form. Since
$$|e^{i\beta_l}+\epsilon_kz|^{2N}\mathfrak{h}_k(\delta_ke_1)=\mathfrak{h}_k(\delta_ke_1)+O(\epsilon_k |z|),$$
and
$$\epsilon_k^2e^{\xi_k(e^{i\beta_l}+\epsilon_kz)}=e^{U_k(z)}+O(\epsilon_k^{\epsilon})(1+|z|)^{-3}. $$
Based on (\ref{t-wk}) we write the equation for $W_k$ as
\begin{equation}\label{around-l}
\Delta W_k(z)+\mathfrak{h}_k(\delta_ke_1)e^{U_k}W_k=E_l^k(z)
\end{equation}
where
$$E_l^k(z)=O(\epsilon_k^{\epsilon})(1+|z|)^{-3}\quad \mbox{in}\quad \Omega_{k,l}.$$
In order to prove $c_0=0$, the key is to control the derivative of $W_0^k(r)$ where
$$W_0^k(r)=\frac 1{2\pi r}\int_{\partial B_r} W_k(re^{i\theta})dS, \quad 0<r<\tau \epsilon_k^{-1}. $$
To obtain a control of $\frac{d}{dr}W_0^k(r)$ we use $\phi_0^k(r)$ as the radial solution of
$$\Delta \phi_0^k+\mathfrak{h}_k(\delta_k e_1)e^{U_k}\phi_0^k=0, \quad \mbox{in }\quad \mathbb R^2. $$

When $k\to \infty$, $\phi_0^k\to c_0\phi_0$. Thus using the equation for $\phi_0^k$ and $W_k$, we have
\begin{equation}\label{c-0-pf}
\int_{\partial B_r}(\partial_{\nu}W_k\phi_0^k-\partial_{\nu}\phi_0^kW_k)=o(\epsilon_k^{\epsilon}). \end{equation}

Thus from (\ref{c-0-pf}) we have
\begin{equation}\label{W-0-d}
\frac{d}{dr}W_0^k(r)=\frac{1}{2\pi r}\int_{\partial B_r}\partial_{\nu}W_k=o(\epsilon_k^{\epsilon})/r+O(1/r^3),\quad 1<r<\tau \epsilon_k^{-1}.
\end{equation}
Since we have known that
$$W_0^k(\tau \epsilon_k^{-1})=o(1). $$
By the fundamental theorem of calculus we have
$$W_0^k(r)=W_0^k(\tau\epsilon_k^{-1})+\int_{\tau \epsilon_k^{-1}}^r(\frac{o(\epsilon_k^{\epsilon})}{s}+O(s^{-3}))ds=O(1/r^2)+O(\epsilon_k^{\epsilon}\log
\frac{1}{\epsilon_k}) $$
for $r\ge 1$. Thus
$c_0=0$ because $W_0^k(r)\to c_0\phi_0$, which means when $r$ is large, it is $-c_0+O(1/r^2)$.

\medskip

\noindent{\bf Step two: $v_k$ is close to a global solution near each $Q_l^k$}.
 We first observe once we have proved $c_1=c_2=c_0=0$ around each $e^{i\beta_l}$, it is easy to use maximum principle to prove $\tilde w_k=o(1)$ in $B_3$ using $\tilde w_k=o(1)$ on $\partial B_3$ and the Green's representation of $\tilde w_k$. The smallness of $\tilde w_k$ immediately implies $\tilde w_k=o(1)$ in $B_R$ for any fixed $R>>1$. Outside $B_R$, a crude estimate of $v_k$ is
  $$v_k(y)\le -\mu_k-4(N+1)\log |y|+C, \quad 3<|y|<\tau \delta_k^{-1}. $$
  Using this and the Green's representation of $w_k$ we can first observe that the oscillation on each $\partial B_r$ is $o(1)$ ($R<r<\tau \delta_k^{-1}/2$) and then by the Green's representation of $\tilde w_k$ and fast decay rate of $e^{V_k}$ we obtain $\tilde w_k=o(1)$ in $\overline{B(0,\tau \delta_k^{-1})}$. A contradiction to $\max |\tilde w_k|=1$.

 There are $N+1$ local maximums with one of them being $e_1$. Correspondingly there are $N+1$ global solutions $V_{l,k}$ that
approximate $v_k$ accurately near $Q_l^k$ for $l=0,...,N$. Note that $Q_0^k=e_1$. For $V_{l,k}$ the expression is
$$V_{l,k}=\log \frac{e^{\mu_l^k}}{(1+\frac{e^{\mu_l^k}}{D_l^k}|y^{N+1}-(e_1+p_l^k)|^2)^2},\quad l=0,...,N, $$
where $p_l^k=E$ and
\begin{equation}\label{def-D}
D_l^k=8(N+1)^2/\mathfrak{h}_k(\delta_kQ_l^k).
\end{equation}
The equation that $V_{l,k}$ satisfies is 
$$\Delta V_{l,k}+|y|^{2N}\mathfrak{h}_k(\delta_k Q_l^k)e^{V_{l,k}}=0,\quad \mbox{in}\quad \mathbb R^2. $$
Since $v_k$ and $V_{l,k}$ have the same common local maximum at $Q_l^k$, it is easy to see that
\begin{equation}\label{ql-exp}
Q_l^k=e^{i\beta_l}+\frac{p_l^ke^{i\beta_l}}{N+1}+O(|p_l^k|^2),\quad \beta_l=\frac{2l\pi}{N+2}.
\end{equation}
Let $M_{l,k}$ be the maximum of $|v_k-V_{l,k}|$ and we claim that all these $M_{l,k}$ are comparable:
\begin{equation}\label{M-comp}
M_{l,k}\sim M_{s,k},\quad \forall s\neq l.
\end{equation}
The proof of (\ref{M-comp}) is as follows: We use $L_{s,l}$ to denote the limit of $(v_k-V_{l,k})/M_{l,k}$ around $Q_s^k$:
\begin{equation}\label{what-linear-1}
\frac{(v_k-V_{l,k})(Q_s^k+\epsilon_kz)}{M_{l,k}}=L_{s,l}+o(1),\quad |z|\le \tau \epsilon_k^{-1} 
\end{equation}
where
$$ L_{s,l}=c_{1,s,l}\frac{z_1}{1+\frac 18 |z|^2}+c_{2,s,l}\frac{z_2}{1+\frac 18 |z|^2},\quad \mbox{and}\quad L_{l,l}=0, \quad s=0,...,N. $$
If all $c_{1,s,l}$ and $c_{2,s,l}$ are zero for a fixed $l$, we can obtain a contradiction just like the beginning of step two. So at least one of them is not zero.
For each $s\neq l$, by Lemma \ref{t-w-1-better} we have
\begin{equation}\label{Q-bad}
v_k(Q_s^k+\epsilon_kz)-V_{s,k}(Q_s^k+\epsilon_kz)=O(\epsilon_k^{\sigma})(1+|z|)^{\sigma} M_{s,k},\quad |z|<\tau \epsilon_k^{-1}.
\end{equation}
Let $M_k=\max_{i}M_{i,k}$ ($i=0,...,N$) and we suppose $M_k=M_{l,k}$. Then to determine $L_{s,l}$ we see that
\begin{align}\label{what-linear-2}
    &\frac{v_k(Q_s^k+\epsilon_k z)-V_{l,k}(Q_s^k+\epsilon_kz)}{M_k}\\
    =&o(\epsilon_k^{\sigma})(1+|z|)^{\sigma}+\frac{V_{s,k}(Q_s^k+\epsilon_kz)-V_{l,k}(Q_s^k+\epsilon_kz)}{M_k}. \nonumber
\end{align}

\noindent{\bf Step 3: Use global solutions to determine $L_{s,l}$.}
(\ref{what-linear-1}) and (\ref{what-linear-2}) provide crucial information to determine the coefficients of $L_{s,l}$. From them we know that $L_{s,l}$ is mainly determined by the difference of two global solutions $V_{s,k}$ and $V_{l,k}$. In order to obtain a contradiction to our assumption we will put the difference in several terms. The main idea in this part of the reasoning is that ``first order terms" tell us what the kernel functions should be, then the ``second order terms" tell us where the pathology is. 

We write $V_{s,k}(y)-V_{l,k}(y)$ as
$$V_{s,k}(y)-V_{l,k}(y)=\mu_s^k-\mu_l^k+2A-A^2+O(|A|^3) $$
where
$$A(y)=\frac{\frac{e^{\mu_l^k}}{D_l^k}|y^{N+1}-e_1-p_l^k|^2-\frac{e^{\mu_s^k}}{D_s^k}|y^{N+1}-e_1-p_s^k|^2}{1+\frac{e^{\mu_s^k}}{D_s^k}|y^{N+1}-e_1-p_s^k|^2}.$$
Here for convenience we abuse the notation $\epsilon_k$ by assuming $\epsilon_k=e^{-\mu_s^k/2}$. Note that $\epsilon_k=e^{-\mu_t^k/2}$ for some $t$, but it does not matter which $t$ it is. From $A$ we claim that 
\begin{align}\label{late-1}
&V_{s,k}(Q_s^k+\epsilon_kz)-V_{l,k}(Q_s^k+\epsilon_kz)\\
=&\phi_1+\phi_2+\phi_3+\phi_4+\mathfrak{R},\nonumber
\end{align}
where
\begin{align*}
&\phi_1=(\mu_s^k-\mu_l^k)\frac{1-\frac{\mathfrak{h}_k(\delta_kQ_s^k)}{8}|z+\frac{N}2\epsilon_k z^2e^{-i\beta_s}+O(\epsilon_k^2)(1+|z|)^3|^2}{\mathrm{B}}, \\
&\phi_2=\frac{\mathfrak{h}_k(\delta_kQ_s^k)}{4 \mathrm{B}} \delta_k\nabla \log \mathfrak{h}_k(\delta_k Q_s^k)(Q_l^k-Q_s^k)\\
&\qquad \cdot \bigg |z+\frac{(p_s^k-p_l^k)e^{i\beta_s}}{(N+1)\epsilon_k}+\frac{N}2\epsilon_kz^2e^{-i\beta_s}+O(\epsilon_k^2(1+|z|)^3 \bigg |^2\\
&\phi_3=\frac{\mathfrak{h}_k(\delta_kQ_s^k)}{2\mathrm{B}}Re((z+\frac{N}2\epsilon_ke^{-i\beta_s}z^2+O(\epsilon_k^2)(1+|z|)^3))(\frac{\bar p_s^k-\bar p_l^k}{(N+1)\epsilon_k}e^{-i\beta_s}))\\
&\phi_4=\frac{\mathfrak{h}_k(\delta_kQ_s)}4\frac{|p_s^k-p_l^k|^2}{(N+1)^2\epsilon_k^2}\frac{1}{\mathrm{B}^2}(1-\frac{\mathfrak{h}_k(\delta_kQ_s)}8|z|^2\cos (2\theta-2\theta_{st}-2\beta_s) ),\\
&\mathrm{B}=1+\frac{\mathfrak{h}_k(\delta_kQ_s^k)}{8}|z+\frac N2\epsilon_ke^{-i\beta_s}z^2+O(\epsilon_k^2(1+|z|)^3)|^2,
\end{align*}
and $\mathfrak{R}_k$ is the collections of other insignificant terms.  Here we briefly explain the roles of each term. $\phi_1$ corresponds to the radial solution in the kernel of the linearized operator of the global equation. In other words, $\phi_1^k/M_k$ should tend to zero because in step one we have proved $c_0=0$. $\phi_2^k/M_k$ is the combination of the two other functions in the kernel. $\phi_4$ is the second order term which will play a leading role later. $\phi_3^k$ comes from the difference of $\mathfrak{h}_k$ at $Q_l^k$ and $Q_s^k$. The derivation of (\ref{late-1}) is as follows: Here we use simplified notations for convenience. First we list the following elementary expressions:
\begin{equation}\label{elem-1}
e^{\mu_l}=e^{\mu_s}(1+(\mu_l-\mu_s)+O((\mu_l-\mu_s)^2)).
\end{equation}
By the definition of $D_s^k$ in (\ref{def-D})
\begin{equation}\label{D-ls}
\frac 1{D_l}=\frac{1}{D_s}(1+\frac{D_s-D_l}{D_l})
=\frac{1}{D_s}(1+\delta_k\nabla \log \mathfrak{h}_k(\delta_kQ_s)(Q_l-Q_s))+O(\delta_k^2)
\end{equation}
\begin{align}\label{elem-2}
&|y^{N+1}-e_1-p_l|^2-|y^{N+1}-e_1-p_s|^2\\
=&|y^{N+1}-1-p_s+(p_s-p_l)|^2-|y^{N+1}-1-p_s|^2 \nonumber\\
=&2 Re\bigg ((y^{N+1}-1-p_s)(\bar p_s-\bar p_l)\bigg )+|p_l-p_s|^2. \nonumber
\end{align}
Using (\ref{elem-1}),(\ref{D-ls}) and (\ref{elem-2}) we have
\begin{align}\label{elem-3}
&\frac{e^{\mu_l}}{D_l}|y^{N+1}-e_1-p_l|^2-\frac{e^{\mu_s}}{D_s}|y^{N+1}-e_1-p_s|^2\\
=&\frac{e^{\mu_s}}{D_s}(2 Re\bigg ((y^{N+1}-1-p_s)(\bar p_s-\bar p_l)\bigg )+|p_s-p_l|^2)\nonumber\\
+&\frac{e^{\mu_s}}{D_s}|y^{N+1}-1-p_l|^2(\delta_k \nabla \log \mathfrak{h}_k(\delta_kQ_s)(Q_l-Q_s)+\mu_l-\mu_s+E_{c,k}) \nonumber
\end{align}
where $E_{c,k}$ is a constant of the size $O((\mu_l-\mu_s)^2)+O(\delta_k^2)$. By the expression of $Q_s^k$ in (\ref{ql-exp}) we have, for $y=Q_s^k+\epsilon_k z$,  
$$y^{N+1}=1+p_s+(N+1)\epsilon_kze^{-i\beta_s}+\frac{N(N+1)}2\epsilon_k^2z^2e^{-2i\beta_s}+O(\epsilon_k^3)(1+|z|)^3, $$
which yields
\begin{equation}\label{elem-4}|y^{N+1}-e_1-p_s|^2=(N+1)^2\epsilon_k^2|z+\frac{N}2\epsilon_ke^{-i\beta_s}z^2+O(\epsilon_k^2)(1+|z|)^3|^2
\end{equation}
\begin{align}\label{elem-5}
&|y^{N+1}-e_1-p_l|^2\\
=&(N+1)^2\epsilon_k^2|z+\frac{(p_s-p_l)e^{i\beta_s}}{(N+1)\epsilon_k} +\frac{N}2\epsilon_k z^2e^{-i\beta_s}+O(\epsilon_k^2)(1+|z|)^3|^2. \nonumber
\end{align}
Using (\ref{elem-4}) and (\ref{elem-5}) in (\ref{elem-3}) we have
\begin{align}\label{exp-A-2}
&A=\frac{\frac{e^{\mu_l^k}}{D_l^k}|y^{N+1}-e_1-p_l|^2-\frac{e^{\mu_s^k}}{D_s^k}|y^{N+1}-e_1-p_s|^2}{1+\frac{e^{\mu_s^k}}{D_s^k}|y^{N+1}-e_1-p_s|^2}\\
&=\frac{\mathfrak{h}_k(\delta_kQ_s)}{8}\bigg (2Re [(z+\frac{N}2\epsilon_k e^{-\beta_s}z^2+O(\epsilon_k^2)(1+|z|)^3)\frac{\bar p_s-\bar p_l}{(N+1)\epsilon_k}e^{-i\beta_s}] \nonumber\\
&+|\frac{p_s-p_l}{(N+1)\epsilon_k}|^2+\bigg |z+\frac{(p_s-p_l)e^{i\beta_s}}{(N+1)\epsilon_k}+\frac N2\epsilon_kz^2e^{-i\beta_s}+O(\epsilon_k^2)(1+|z|)^3\bigg|^2 *\nonumber\\
&\,\,  (\delta_k\nabla(\log \mathfrak{h}_k)(\delta_kQ_s)(Q_l-Q_s)+\mu_l-\mu_s+E_{c,k})\bigg )/\mathrm{B}. \nonumber
\end{align}
where the expression of $\mathrm{B}$ is
\[\mathrm{B}=1+\frac{\mathfrak{h}_k(\delta_kQ_s^k)}{8}|z+\frac N2\epsilon_ke^{-i\beta_s}z^2+O(\epsilon_k^2(1+|z|)^3)|^2.\]
Here we point out the expression of $\phi_1$ is a combination of the $\mu_s^k-\mu_l^k$ outside $2A-A^2$ and the $(\mu_l^k-\mu_s^k)$ term in the expression of $A$ in (\ref{exp-A-2}).

For $A^2$ the leading term, which is the only term that matters in the computation later is 
\begin{equation}\label{exp-a-squ}
A^2=
\bigg (\frac{\mathfrak{h}_k(\delta_kQ_s)^2}{32}|z|^2 |\frac{p_s-p_l}{(N+1)\epsilon_k}|^2(1+\cos(2\theta-2\theta_{sl}-\beta_s))\bigg )/\mathrm{B}^2+\mathfrak{R}
\end{equation}
where $z=|z|e^{i\theta}$, $p_s-p_l=|p_s-p_l|e^{i\theta_{sl}}$ and $\mathfrak{R}$ represents the sum of other terms.
Using these expressions we can obtain (\ref{late-1}) by direct computation. 
Here $\phi_1$, $\phi_3$ correspond to solutions to the linearized operator. Here we note that if we set $\epsilon_{l,k}=e^{-\mu_l^k/2}$, there is no essential difference between $\epsilon_{l,k}$ and $\epsilon_k=e^{-\frac 12\mu_{1,k}}$ because $\epsilon_{l,k}=\epsilon_k(1+o(1))$. If $|\mu_{s,k}-\mu_{l,k}|/M_k\ge C$ we get a contradiction to $\tilde w_k=o(1)$ outside the bubble disks. Thus, we must have $|\mu_{s,k}-\mu_{l,k}|/M_k\to 0$. After simplification (see $\phi_3$ of (\ref{late-1})) we have
\begin{align}\label{c-12}
c_{1,s,l}=\lim_{k\to \infty}\frac{|p_s^k-p_l^k|}{2(N+1)M_k\epsilon_k}\cos(\beta_s+\theta_{sl}),\\
c_{2,s,l}=\lim_{k\to \infty}
\frac{|p_s^k-p_l^k|}{2(N+1)\epsilon_k M_k}\sin(\beta_s+\theta_{sl})\nonumber
\end{align}
It is also important to observe that even if $M_k=o(\epsilon_k)$ we still have $M_k\sim \max_{s}|p_s^k-p_l^k|/\epsilon_k$. Since each $|p_l^k|=E$, an upper bound for $M_k$ is 
\begin{equation}\label{upper-bound-Mk}
M_k\le C\mu_k\epsilon_k+C\delta_k^2\epsilon_k^{-1}\le C\mu_k\epsilon_k.
\end{equation}

Equation (\ref{c-12}) gives us a key observation: $|c_{1,s,l}|+|c_{2,s,l}|\sim |p_s^k-p_l^k|/(\epsilon_k M_k)$. So whenever $|c_{1,s,l}|+|c_{2,s,l}|\neq 0$ we have
$\frac{|p_s^k-p_l^k|}{\epsilon_k}\sim M_k$. In other words for each $l$, $M_{l,k}\sim \max_{t\neq l}\frac{|p_t^k-p_l^k|}{\epsilon_k}$.  Hence for any $t$, if $\frac{|p_t^k-p_l^k|}{\epsilon_k}\sim M_k$, let $M_{t,k}$ be the maximum of $|v_k-V_{t,k}|$, we have $M_{t,k}\sim M_k$. If all $\frac{|p_t^k-p_l^k|}{\epsilon_k}\sim M_k$ (\ref{M-comp}) is proved. So we prove that even if some $p_t^k$ is very close to $p_l^k$, $M_t^k$ is still comparable to $M_k$. Here is the reason, without loss of generality, $M_k=M_{1,k}$ and corresponding to $M_k$, there exist $p_{s}^k$ such that 
$$\frac{|p_{1}^k-p_{s}^k|}{\epsilon_k}\sim M_k.$$
For any given $p_t^k$, if $p_t^k$ is too close to $p_1^k$: $|p_t^k-p_{1}^k|<\frac 15|p_s^k-p_1^k|$, then $|p_t^k-p_s^k|\ge \frac 12 |p_s^k-p_1^k|$. 
Thus $\frac{|p_t^k-p_{s}^k|}{\epsilon_k}\sim M_k$ and $M_t^k\sim M_k$. (\ref{M-comp}) is established.  From now on for convenience we shall just use $M_k$. Since 
$M_k\sim \max_{s,t}|p_s^k-p_t^k|/\epsilon_k$, (\ref{upper-bound-Mk}) holds for $M_k$.

Now we set 
\begin{equation}\label{wlk-def}
w_{l,k}=(v_k-V_{l,k}).
\end{equation}
and 
$$\tilde w_{l,k}=w_{l,k}/M_k.$$

The equation of $w_{l,k}$ can be written as 
\begin{align}\label{wlk-bs}
   &\Delta w_{l,k}+|y|^{2N}\mathfrak{h}_k(\delta_k Q_l)e^{\xi_l}w_{l,k}\\
=&-\delta_k\nabla\mathfrak{h}_k(\delta_kQ_l)(y-Q_l)|y|^{2N}e^{V_{l,k}}-\delta_k^2
\sum_{|\alpha |=2}\frac{\partial^{\alpha}\mathfrak{h}_k(\delta_kQ_l)}{\alpha !}(y-Q_l)^{\alpha}|y|^{2N}e^{V_{l,k}}\nonumber\\
&+O(\delta_k^3)|y-Q_l|^3|y|^{2N}e^{V_{l,k}} \nonumber 
\end{align}
where we omitted $k$ in $Q_l$ and $\xi_l$. $\xi_l$ comes from the Mean Value Theorem and satisfies
\begin{equation}\label{around-ls}
e^{\xi_l}=e^{V_{l,k}}(1+\frac 12w_{l,k}+O(w_{l,k}^2)). 
\end{equation}
The function $\tilde w_{l,k}$ satisfies
\begin{equation}\label{aroud-s-1}
\lim_{k\to \infty}\tilde w_{l,k}(Q_s^k+\epsilon_k z)=\frac{c_{1,s,l}z_1+c_{2,s,l}z_2}{1+\frac 18 |z|^2}
\end{equation}
and around each $Q_s^k$ (\ref{Q-bad}) holds with $M_{s,k}$ replaced by $M_k$. The equation of $\tilde w_{l,k}$ is 
\begin{align}\label{t-wlk}
&\Delta \tilde w_{l,k}+|y|^{2N}\mathfrak{h}_k(\delta_kQ_l^k)e^{\xi_l^k}\tilde w_{l,k}\\
=&o(1)(y-Q_l^k)|y|^{2N}e^{V_{l,k}}+o(1)\sum_{|\alpha |=2}\frac{\partial^{\alpha}\mathfrak{h}_k(\delta_kQ_l^k)}{\alpha !}(y-Q_l^k)^{\alpha}|y|^{2N}e^{V_{l,k}}\nonumber \\
&+o(\delta_k)|y-Q_l^k|^3|y|^{2N}e^{V_{l,k}} \nonumber
\end{align}

\noindent{\bf Step four: Better estimate of $\tilde w_{l,k}$ away from local maximums.}
Now for $|y|\sim 1$, we use $w_{l,k}(Q_l^k)=0$ to write $w_{l,k}(y)$ as
\begin{align}\label{extra-2}
    w_{l,k}(y)&=\int_{\Omega_k}(G_k(y,\eta)-G_k(Q_l,\eta))\bigg (\mathfrak{h}_k(\delta_k Q_l)|\eta |^{2N}e^{\xi_l}w_{l,k}(\eta)\\
    &+\delta_k \nabla \mathfrak{h}_k(\delta_k Q_l)(\eta-Q_l)|\eta |^{2N}e^{V_{l,k}}\nonumber\\
    &+\delta_k^2
\sum_{|\alpha |=2}\frac{\partial^{\alpha}\mathfrak{h}_k(\delta_kQ_l)}{\alpha !}(\eta-Q_l)^{\alpha}|\eta |^{2N}e^{V_{l,k}}\bigg )
+o(\delta_k^{2}). \nonumber
\end{align}
Note that the 
the oscillation of $w_{l,k}$ on $\partial \Omega_k$ is $O(\delta_k^{N+1})$. The harmonic function defined by the boundary value of $w_{l,k}$ has an oscillation of  $O(\delta_k^{N+1})$ on $\partial \Omega_k$. The oscillation of this harmonic function in $B_R$ (for any fixed $R>1$) is $O(\delta_k^{N+2})$.
The regular part of the Green's function brings little error in the computation, indeed
\begin{align*}
&G_k(y,\eta)-G_k(Q_l^k,\eta)\\
=&\frac 1{2\pi}\log\frac{|Q_l-\eta |}{|y-\eta |}+\frac 1{2\pi}
\log \bigg |\frac{\frac{\eta }{|\eta |}-L_k^{-2}y|\eta |}{\frac{\eta }{|\eta |}-L_k^{-2}Q_l|\eta |}\bigg |.\\
=&\frac 1{2\pi}\log\frac{|Q_l-\eta |}{|y-\eta |}+O(\delta_k^2)|y||\eta |,\quad \mbox{for}\quad |y|\sim 1
\end{align*} 
where $L_k=\tau \delta_k^{-1}$. 
When we consider the integration in (\ref{extra-2}) from the last term, we have 
the order is $O(\delta_k^2)$, which is $o(\epsilon_k^2\mu_k)=o(\epsilon_k)$.  So we have 
 
 \begin{align}\label{green-wlk}
& \tilde w_{l,k}(y)
=-\frac{1}{2\pi}\int_{\Omega_k}\log \frac{|y-\eta |}{|Q_l^k-\eta |}\bigg ( \tilde w_{l,k}(\eta )\mathfrak{h}_k(\delta_k Q_l^k)|\eta |^{2N}e^{\xi_l}\\
& +\sigma_k\nabla \mathfrak{h}_k(\delta_k Q_l^k)(\eta-Q_l^k)
|\eta |^{2N}e^{V_{l,k}}\nonumber\\
&+\frac{\delta_k^2}{M_k}\sum_{|\alpha |=2}\frac{\partial^{\alpha}\mathfrak{h}_k(\delta_k Q_l^k)}{\alpha !}(\eta -Q_l^k)^{\alpha}|\eta |^{2N}e^{V_{l,k}}\bigg )d\eta
+o(\epsilon_k),\nonumber\\
&=\tilde H_{l,k}+o(\epsilon_k)\quad \mbox{for} \quad |y|\sim 1. \nonumber
\end{align}

Here we note that it is important to make the error $o(\epsilon_k)$. A larger error than $o(\epsilon_k)$ would cause major problems. Now we claim a better estimate of $\tilde w_{l,k}$ around $Q_l^k$: 
\begin{equation}\label{small-q}
|\tilde w_{l,k}(Q_l^k+\epsilon_ky)|=o(\epsilon_k)(1+|y|),\quad |y|\le \tau \epsilon_k^{-1}.
\end{equation}

\noindent{\bf Proof of (\ref{small-q}):} First we need 
 the following crude identity based on (\ref{late-1}):
\begin{align}\label{late-2-r}
\int_{B(Q_s^k,\tau)}(\tilde w_{l,k}(\eta)\mathfrak{h}_k(\delta_k Q_l^k)|\eta |^{2N}e^{V_{s,k}}
+\sigma_{k}\nabla \mathfrak{h}_k(\delta_k Q_l^k)(\eta -Q_l^k)|\eta |^{2N}e^{V_{l,k}})d\eta \nonumber\\
=O(\epsilon_k^{\delta}),\quad s=0,...,N,
\end{align}
for some $\delta\in (0,1)$.
When we compare the first term on the right hand side of (\ref{green-wlk}) and the first term of (\ref{late-2-r}), $e^{\xi_l}$ is replaced by $e^{V_{s,k}}$, this replacement is minor, as one can check from (\ref{upper-bound-Mk}) and (\ref{around-ls}) that 
\[ | e^{\xi_l(Q_s^k+\epsilon_kz)}-e^{V_{s,k}(Q_s^k+\epsilon_kz)}|
=o(\epsilon_k^{\delta})(1+|z|)^{-4},\quad |z|\le \tau_1 \epsilon_k^{-1}. \]
(\ref{late-1}) is mainly used in the evaluation of the first term. In order not to disturb the main stream of the proof, we put 
the proof of (\ref{late-2-r}) at the end of this step. The reason it is called a crude estimate is because its actual leading term is of the order $O(\epsilon_k)$, but it is sufficient to have $O(\epsilon_k^{\delta})$ for the proof of (\ref{small-q}). Next we observe from (\ref{green-wlk}) that $\tilde H_{l,k}$ is a harmonic function in $B(Q_l^k,\tau_1)$ and $\tilde H_{l,k}(Q_l^k)=0$. Now we evaluate $H_{l,k}$ on $|y-Q_l^k|=\tau_1$. It is easy to see that the integral outside $\cup_{s=0}^NB(Q_s^k,\tau_1)$ is $O(\epsilon_k^2)$. Next we see that the integral over $B(Q_l^k,\tau_1)$ is $o(\epsilon_k^{\delta})$ for some $\delta>0$ because for the first term and the second term we use (\ref{late-2-r}), for the third terms we have $|\eta -Q_l^k|^2$, which contributes $\epsilon_k^2$ after scaling. The integration over other disks gives
\[\sum_{s\neq l}\sigma_{s,k}\log \frac{|y-Q_s^k|}{|Q_l^k-Q_s^k|}\]
where $\sigma_{s,k}\to 0$ as $k\to \infty$. Since this is a harmonic function we have 
\[\tilde H_{l,k}(Q_l^k+\epsilon_kz)=O(\tilde \sigma_k)\epsilon_k(1+|z|),\quad |z|\le \tau_1 \epsilon_k^{-1},\]
for some $\tilde \sigma_k\to 0$. Thus we have (\ref{small-q}).

At the end of this step we prove (\ref{late-2-r}).

 Here we recall that $v_k$ is close to $V_{s,k}$ near $Q_s^k$ (see \ref{Q-bad})). That is why we shall use (\ref{late-1}). We state (\ref{late-2-r}) here:
\begin{align*}
\int_{B(Q_s^k,\tau)}(\tilde w_{l,k}(\eta)\mathfrak{h}_k(\delta_k Q_l^k)|\eta |^{2N}e^{V_{s,k}}
+\sigma_{k}\nabla \mathfrak{h}_k(\delta_k Q_l^k)(\eta -Q_l^k)|\eta |^{2N}e^{V_{l,k}})d\eta\\
=O(\epsilon_k^{\delta}) 
\end{align*}
for some $\delta\in (0,1)$.
Before the evaluation we recall the definition of $w_{l,k}$ in (\ref{wlk-def}) that around $Q_s^k$, 
$$\tilde w_{l,k}=\frac{v_k-V_{s,k}}{M_k}+\frac{V_{s,k}-V_{l,k}}{M_k}.$$
The first term after scaling at $Q_s^k$ is $O(\epsilon_k^{\delta})(1+|y|)^{\delta})$, so the leading term in the second term. So our goal in this section is to prove 
\begin{align}\label{crucial-8}
\int_{B(Q_s^k,\tau)}(\frac{V_{s,k}-V_{l,k}}{M_k}\mathfrak{h}_k(\delta_k Q_l^k)|\eta |^{2N}e^{V_{s,k}}
+\sigma_{k}\nabla \mathfrak{h}_k(\delta_k Q_l^k)(\eta -Q_l^k)|\eta |^{2N}e^{V_{l,k}})d\eta \nonumber\\
=O(\epsilon_k^{\delta}) 
\end{align}
Then we use (\ref{late-1}) in $(V_{s,k}-V_{l,k})/M_k$. 
Later we shall see that the terms of $\phi_1$ and $\phi_3$ lead to $o(\epsilon_k)$. We first look at the integration involving $\phi_2$: 
\begin{align}\label{phi-2-com}
 & \int_{B(Q_s^k,\tau)}\frac{\phi_2^k}{M_k}\mathfrak{h}_k(\delta_k Q_l^k)|\eta |^{2N}e^{V_{s,k}}
d\eta  \\
=&\frac{\mathfrak{h}_k(\delta_kQ_s^k))}4 \sigma_k\nabla \mathfrak{h}_k(\delta_kQ_s^k)(Q_l^k-Q_s^k)\int_{B(Q_s^k,\tau)}\frac{|z|^2}{(1+\frac{\mathfrak{h}_k(\delta_kQ_l^k)}8|z|^2)^3}dz \nonumber\\
=&8\pi \sigma_k\nabla (\log \mathfrak{h}_k(\delta_kQ_s^k)(Q_l^k-Q_s^k)\big (1+O(\epsilon_k^2\log 1/\epsilon_k)\big )\nonumber
\end{align}
We see that this term almost
cancels with the second term of (\ref{late-2-r}). The computation of $\phi_2$ is based on this equation:
\begin{equation}\label{phi-2-more}\int_{\mathbb R^2}\frac{\frac{\mathfrak{h}_k(\delta_k Q_s^k)}4\sigma_k\nabla \mathfrak{h}_k(\delta_k Q_s^k)(Q_l^k-Q_s^k)|z|^2}{(1+\frac{\mathfrak{h}_k(\delta_k Q_s^k)}8|z|^2)^3}dz
=8\pi \sigma_k\nabla (\log \mathfrak{h}_k)(\delta_k Q_s^k)(Q_l^k-Q_s^k), 
\end{equation}
and by (\ref{delta-small-1})
\begin{equation}\label{subtle-1}
\nabla \log \mathfrak{h}_k(\delta_kQ_l^k)-\nabla \log \mathfrak{h}_k(\delta_k Q_s^k)=O(\delta_k)=o(\epsilon_k\mu_k^{\frac 12}). 
\end{equation}

The integration involving $\phi_4$ provides the leading term. More detailed information is the following:
First for a global solution
$$V_{\mu,p}=\log \frac{e^{\mu}}{(1+\frac{e^{\mu}}{\lambda}|z^{N+1}-p|^2)^2}$$ of
$$\Delta V_{\mu,p}+\frac{8(N+1)^2}{\lambda}|z|^{2N}e^{V_{\mu,p}}=0,\quad \mbox{in }\quad \mathbb R^2, $$
by differentiation with respect to $\mu$ we have
$$\Delta(\partial_{\mu}V_{\mu,p})+\frac{8(N+1)^2}{\lambda}|z|^{2N}e^{V_{\mu,p}}\partial_{\mu}V_{\mu,p}=0,\quad \mbox{in}\quad \mathbb R^2. $$
By the expression of $V_{\mu,p}$ we see that 
$$\partial_r\bigg (\partial_{\mu}V_{\mu,p}\bigg )(x)=O(|x|^{-2N-3}).$$ 
Thus we have
\begin{equation}\label{inte-eq-1}
\int_{\mathbb R^2}\partial_{\mu}V_{\mu,p}|z|^{2N}e^{V_{\mu,p}}=\int_{\mathbb R^2}\frac{(1-\frac{e^{\mu}}{\lambda}|z^{N+1}-P|^2)|z|^{2N}}{(1+\frac{e^{\mu}}{\lambda}|z^{N+1}-P|^2)^3}dz=0.
\end{equation}

From $V_{\mu,p}$  we also have
$$\int_{\mathbb R^2}\partial_{P}V_{\mu,p}|y|^{2N}e^{V_{\mu,p}}=\int_{\mathbb R^2}\partial_{\bar P}V_{\mu,p}|y|^{2N}e^{V_{\mu,p}}=0, $$
which gives
\begin{equation}\label{inte-eq-2}
\int_{\mathbb R^2}\frac{\frac{e^{\mu}}{\lambda}(\bar z^{N+1}-\bar P)|z|^{2N}}{(1+\frac{e^{\mu}}{\lambda}|z^{N+1}-P|^2)^3}=\int_{\mathbb R^2}\frac{\frac{e^{\mu}}{\lambda}( z^{N+1}- P)|z|^{2N}}{(1+\frac{e^{\mu}}{\lambda}|z^{N+1}-P|^2)^3}=0.
\end{equation}

From (\ref{inte-eq-1}) and (\ref{inte-eq-2})  we use scaling and cancellation to have
\begin{equation}\label{small-phi-13}\int_{B(0,\tau\epsilon_k^{-1})}\frac{\phi_1}{M_k}B^{-2}=o(\epsilon_k),\quad
\int_{B(0,\tau\epsilon_k^{-1})}\frac{\phi_3}{M_k}B^{-2}=o(\epsilon_k).
\end{equation}
Thus (\ref{late-2-r}) holds.

\noindent{\bf Step five: Completion of the proof.}
We recall from (\ref{wlk-def}) that around $Q_s^k$
\begin{equation}\label{decom-s}\tilde w_{l,k}=\frac{v_k-V_{s,k}}{M_k}+\frac{V_{s,k}-V_{l,k}}{M_k}.
\end{equation}
We define this quantity without giving a precise estimate of it:
\begin{align*}
D_{s,l}^k:=&\int_{B(Q_s^k,\tau)}\bigg ( \tilde w_{l,k}(\eta )\mathfrak{h}_k(\delta_k Q_l^k)|\eta |^{2N}e^{\xi_l}
+\sigma_k\nabla \mathfrak{h}_k(\delta_k Q_l^k)(\eta-Q_l^k)
|\eta |^{2N}e^{V_{l,k}}\\
&+\frac{\delta_k^2}{M_k}\sum_{|\alpha |=2}
\frac{\partial^{\alpha}\mathfrak{h}_k(\delta_k Q_l^k)}{\alpha !}(\eta-Q_l^k)^{\alpha}|\eta |^{2N}e^{V_{l,k}}\bigg )d\eta.
\end{align*}
By (\ref{small-q}) we know $D_{s,l}^k=O(\epsilon_k^{\delta})$. Next we let
$$  H_{y,l}(\eta)=\frac{1}{2\pi}\log\frac{|y-\eta|}{|Q_l^k-\eta |}. $$
Then in (\ref{green-wlk}) we have
\begin{align*}
    &\tilde w_{l,k}(y)=-\sum_{s\neq l}H_{y,l}(Q_s)D_{s,l}^k \\
   & -\sum_{s\neq l}\int_{B(Q_s,\tau)}\big (\partial_1H_{y,l}(Q_s)(\eta_1-Q_s^1)+\partial_2H_{y,l}(Q_s)(\eta_2-Q_s^2)\big )
    \cdot \mathfrak{h}_k(\delta_kQ_l)|\eta |^{2N}e^{\xi_l}\tilde w_{l,k}(\eta)\\
    &+o(\epsilon_k).
\end{align*} 

After evaluation we have
\begin{align*}
    \tilde w_{l,k}(y)&=-\frac{1}{2\pi}\sum_{s\neq l}\log\frac{|y-Q_s^k|}{|Q_l^k-Q_s^k|}D_{s,l}^k
    +\sum_{s\neq l}\bigg (8(\frac{y_1-Q_s^1}{|y-Q_s|^2}-\frac{Q_l^1-Q_s^1}{|Q_l-Q_s|^2})c_{1,s,l}\\
   &+8(\frac{y_2-Q_s^2}{|y-Q_s|^2}-\frac{Q_l^2-Q_s^2}{|Q_s-Q_l|^2})c_{2,s,l}\bigg )\epsilon_k+o(\epsilon_k).
\end{align*}
where we used
$$\int_{\mathbb R^2}\frac{z_1^2}{(1+\frac 18|z|^2)^3}dz=\int_{\mathbb R^2}\frac{z_2^2}{(1+\frac 18|z|^2)^3}dz=16\pi. $$
Recall that $c_{1,s,l}$ and $c_{2,s,l}$ are defined in (\ref{c-12}).

For $|y|\sim 1$ but away from the $N+1$ bubbling disks, we have, for $l\neq s$, 
$$v_k(y)=V_{l,k}(y)+M_k\tilde w_{l,k}(y) $$
and
$$v_k(y)=V_{s,k}(y)+M_k\tilde w_{s,k}(y). $$
Thus for $s\neq l$ we have
\begin{equation}\label{compare-10}
\frac{V_{s,k}(y)-V_{l,k}(y)}{M_k}=\tilde w_{l,k}(y)-\tilde w_{s,k}(y). \end{equation}
In (\ref{late-1}) we consider $|z|\sim \epsilon_k^{-1}$, then we see that if 
\[|\frac{\mu_l^k-\mu_s^k}{M_k}+2\sigma_k\nabla \log \mathfrak{h}_k(\delta_kQ_s^k)(Q_l^k-Q_s^k)|\ge C\epsilon_k,\]
for a large $C$,
 it is easy to see that  (\ref{compare-10}) does not hold. So we have 
\[|\frac{\mu_l^k-\mu_s^k}{M_k}+2\sigma_k\nabla \log \mathfrak{h}_k(\delta_kQ_s^k)(Q_l^k-Q_s^k)|=O(\epsilon_k),\]
and we focus on the 
leading term $\phi_3$, which gives, for $|y|\sim 1$ away from bubbling disks, 
\[\frac{V_{s,k}(y)-V_{l,k}(y)}{M_k}
=D_{\mu}^k\epsilon_k+4Re(\frac{y^{N+1}-1}{|y^{N+1}-1|^2}\frac{\bar p_s-\bar p_l}{M_k\epsilon_k}) \epsilon_k+O(|y-Q_s^k|^2)\epsilon_k+o(\epsilon_k), \]
where 
\[D_{\mu}^k=\frac{\mu_l^k-\mu_s^k}{M_k\epsilon_k}+2\frac{\sigma_k}{\epsilon_k}\nabla \log \mathfrak{h}_k(\delta_kQ_s^k)(Q_l^k-Q_s^k).\]
On the other hand, for $y\in B_5\setminus (\cup_{t=1}^NB(Q_t^k,\tau_1))$, 
\begin{align*}
 &\tilde w_{l,k}(y)-\tilde w_{s,k}(y)\\
 =&-\frac{1}{2\pi}\sum_{m,m\neq l}\log \frac{|y-Q_m|}{|Q_l-Q_m|}D_{m,l}^k\\
 &+8\epsilon_k\sum_{m,m\neq l}\bigg ((\frac{y_1-Q_m^1}{|y-Q_m|^2}-\frac{Q_l^1-Q_m^1}{|Q_l-Q_m|^2})\frac{|p_m-p_l|}{2(N+1)M_k\epsilon_k}\cos(\beta_m+\theta_{ml})\\
 &+(\frac{y_2-Q_m^2}{|y-Q_m|^2}-\frac{Q_l^2-Q_m^2}{|Q_l-Q_m|^2})\frac{|p_m-p_l|}{2(N+1)M_k\epsilon_k}\sin (\beta_m+\theta_{ml})\bigg )\\
 &+\frac{1}{2\pi}\sum_{m,m\neq s}\log \frac{|y-Q_m|}{|Q_s-Q_m|}D_{m,s}^k\\
 &-8\epsilon_k\sum_{m,m\neq s}\bigg ((\frac{y_1-Q_m^1}{|y-Q_m|^2}-\frac{Q_s^1-Q_m^1}{|Q_s-Q_m|^2})\frac{|p_m-p_s|}{2(N+1)M_k\epsilon_k}\cos(\beta_m+\theta_{ms})\\
 &+(\frac{y_2-Q_m^2}{|y-Q_m|^2}-\frac{Q_s^2-Q_m^2}{|Q_s-Q_m|^2})\frac{|p_m-p_s|}{2(N+1)M_k\epsilon_k}\sin (\beta_m+\theta_{ms})\bigg )
\end{align*}
for all $l\neq s$. If we fix a set of $l,s$ that corresponds to the largest $|D_{s,l}^k|$ and we consider $y$ close to $Q_s^k$. If we use $y=e^{i\beta_s}+z$ by abusing the notation $z$, then we have
\[y^{N+1}=(e^{i\beta_s}(1+ze^{-i\beta_s}))^{N+1}=1+(N+1)ze^{-i\beta_s}+O(|z|^2).\]
Therefore
\begin{align*}
&4Re(\frac{y^{N+1}-1}{|y^{N+1}-1|^2}\frac{\bar p_s-\bar p_l}{M_k\epsilon_k}\\
=&\frac{4|p_s-p_l|}{(N+1)|z|^2M_k\epsilon_k}\bigg (z_1\cos(\beta_s+\beta_{sl})+z_2\sin(\beta_s+\beta_{sl})+O(|z|^2)\bigg ).
\end{align*}
In the expression of $\tilde w_{l,k}(y)-\tilde w_{s,k}(y)$, we identify the leading term, which is 
\begin{align*}
8\epsilon_k\bigg ((\frac{y_1-Q_s^1}{|y-Q_s|^2}-\frac{Q_l^1-Q_s^1}{|Q_l-Q_s|^2})\frac{|p_s-p_l|}{2(N+1)M_k\epsilon_k}\cos(\beta_s+\theta_{sl})\\
+(\frac{y_2-Q_s^2}{|y-Q_s|^2}-\frac{Q_l^2-Q_s^2}{|Q_l-Q_s|^2})\frac{|p_s-p_l|}{2(N+1)M_k\epsilon_k}\sin(\beta_s+\theta_{sl})\bigg )\\
-\frac{1}{2\pi}\log \frac{|y-Q_s|}{|Q_l-Q_s|}D_{s,l}^k.
\end{align*}
If we use $y=e^{i\beta_s}+z$ for $|z|$ small and replace $Q_s$ by $e^{i\beta_s}$ because their difference is $o(\epsilon_k)$. Then the expression above has this leading term:
\[\frac{4|p_s-p_l|}{(N+1)|z|^2M_k\epsilon_k}\bigg (z_1\cos(\beta_s+\beta_{sl})+z_2\sin(\beta_s+\beta_{sl})\bigg )-\frac 1{2\pi}\log \frac{|z|}{| e^{i\beta_l}-e^{i\beta_s}|}D_{s,l}^k.\]
Thus we obtain $D_{s,l}^k/\epsilon_k=o(1)$. Therefore 
\begin{equation}\label{dsl0}
D_{s,l}^k=o(\epsilon_k),\quad \forall s\neq l.
\end{equation}
With this updated information we write $\tilde w_{l,k}(y)-\tilde w_{s,k}(y)$ as
\begin{align*}
 &\tilde w_{l,k}(y)-\tilde w_{s,k}(y)\\
 &=8\epsilon_k\sum_{m,m\neq l}\bigg ((\frac{y_1-\cos \beta_m}{|y-e^{i\beta_m}|^2}-\frac{\cos \beta_l-\cos \beta_m}{|e^{i\beta_l}-e^{i\beta_m}|^2})\frac{|p_m-p_l|}{2(N+1)M_k\epsilon_k}\cos(\beta_m+\theta_{ml})\\
 &+(\frac{y_2-\sin \beta_m}{|y-e^{i\beta_m}|^2}-\frac{\sin\beta_l-\sin \beta_m}{|e^{i\beta_l}-e^{i\beta_m}|^2})\frac{|p_m-p_l|}{2(N+1)M_k\epsilon_k}\sin (\beta_m+\theta_{ml})\bigg )\\
 &-8\epsilon_k\sum_{m,m\neq s}\bigg ((\frac{y_1-\cos \beta_m}{|y-e^{i\beta_m}|^2}-\frac{\cos \beta_s-\cos \beta_m}{|e^{i\beta_s}-e^{i\beta_m}|^2})\frac{|p_m-p_s|}{2(N+1)M_k\epsilon_k}\cos(\beta_m+\theta_{ms})\\
 &+(\frac{y_2-\sin \beta_m}{|y-e^{i\beta_m}|^2}-\frac{\sin \beta_s-\sin \beta_m}{|e^{i\beta_s}-e^{i\beta_m}|^2})\frac{|p_m-p_s|}{2(N+1)M_k\epsilon_k}\sin (\beta_m+\theta_{ms})\bigg )+o(\epsilon_k)
\end{align*}
for $|y|\in B_5\setminus (\cup_{l=1}^N B(Q_l^k,\tau_1))$. 

\medskip

Now in particular we take $l=0$ and we use the following notations: $\tilde w_k$, $V_k$, $c_{1,s}$, $c_{2,s}$, $\theta_s$, instead of $\tilde w_{0}^k$, $v_{0,k}$, $c_{1,s,0}$, $c_{2,s,0}$, $\theta_{s,0}$. 

The expression of $\tilde w_k$ (see (\ref{green-wlk}) for example) gives
\begin{align*}
&\nabla \tilde w_k(y)\\
=&\int_{\Omega_k} \nabla_y G(y,\eta)\bigg (\mathfrak{h}_k(\delta_ke_1)|\eta |^{2N}e^{\xi_k}\tilde w_k(\eta)
+\sigma_k\nabla\mathfrak{h}_k(\delta_ke_1)(\eta-e_1)|\eta |^{2N}e^{V_k(\eta)}\\
&+\frac{\delta_k^2}{M_k}\sum_{|\alpha |=2}\frac{\partial^{\alpha}\mathfrak{h}_k(\delta_ke_1)(\eta -e_1)^{\alpha}}{\alpha !}|\eta |^2{2N}e^{V_k(\eta)}\bigg )d\eta
+o(\epsilon_k), 
\end{align*}
for $y\in B_5\setminus (\cup_{s=1}^N B(Q_s^k,\tau_1))$.
Now we take $y=e_1$, we have 
\begin{align*}
&0=\nabla \tilde w_k(e_1)\\
=&\int_{\Omega_k} (-\frac{1}{2\pi})\frac{e_1-\eta}{|e_1-\eta |^2}\bigg (\mathfrak{h}_k(\delta_ke_1)|\eta |^{2N}e^{\xi_k}\tilde w_k(\eta)
+\sigma_k\nabla\mathfrak{h}_k(\delta_ke_1)(\eta-e_1)|\eta |^{2N}e^{V_k(\eta)}\\
&+\frac{\delta_k^2}{M_k}\sum_{|\alpha |=2}\frac{\partial^{\alpha}\mathfrak{h}_k(\delta_ke_1)(\eta -e_1)^{\alpha}}{\alpha !}|\eta |^2{2N}e^{V_k(\eta)}\bigg )d\eta
+o(\epsilon_k), 
\end{align*}
Obviously we will integral each of the two components in $B(Q_s^k,\tau_1)$ for $\tau_1>0$ small. Then we observe from (\ref{dsl0}) 
that
\begin{align*}\int_{B(Q_l,\tau_1)}\bigg (\mathfrak{h}_k(\delta_ke_1)|\eta |^{2N}e^{\xi_k}\tilde w_k(\eta)
+\sigma_k\nabla\mathfrak{h}_k(\delta_ke_1)(\eta-e_1)|\eta |^{2N}e^{V_k(\eta)}\\
+\frac{\delta_k^2}{M_k}\sum_{|\alpha |=2}\frac{\partial^{\alpha}\mathfrak{h}_k(\delta_ke_1)(\eta -e_1)^{\alpha}}{\alpha !}|\eta |^2{2N}e^{V_k(\eta)}\bigg )d\eta=o(\epsilon_k).\end{align*}
Based on this we use the following format: If $f$ is a smooth function,
\begin{align*}
&\int_{B(Q_s,\tau_1)} f(\eta) \bigg (\mathfrak{h}_k(\delta_ke_1)|\eta |^{2N}e^{\xi_k}\tilde w_k(\eta)
+\sigma_k\nabla\mathfrak{h}_k(\delta_ke_1)(\eta-e_1)|\eta |^{2N}e^{V_k(\eta)}\\
&\qquad +\frac{\delta_k^2}{M_k}\sum_{|\alpha |=2}\frac{\partial^{\alpha}\mathfrak{h}_k(\delta_ke_1)(\eta -e_1)^{\alpha}}{\alpha !}|\eta |^2{2N}e^{V_k(\eta)}\bigg )d\eta\\
&=\partial_1f(e^{i\beta_s})c_{1s}\cdot 16\pi \epsilon_k+\partial_1 f(e^{i\beta_s})c_{2s}\cdot 16\pi \epsilon_k +o(\epsilon_k). 
\end{align*}
Then we replace $f(\eta_1,\eta_2)$ by \[f_1(\eta_1,\eta_2)=(-\frac{1}{2\pi})\frac{1-\eta_1}{(1-\eta_1)^2+\eta_2^2}\] and 
\[f_2(\eta_1,\eta_2)=\frac{1}{2\pi}\frac{\eta_2}{(1-\eta_1)^2+\eta_2^2}.\]
Then we have, from the expressions of $c_{1,s}$, $c_{2,s}$ in (\ref{c-12}), that 
\begin{align*}
0&=\partial_1\tilde w_k(e_1)\\
&=16\pi\epsilon_k\sum_{s=1}^N\bigg (\frac{\cos(\beta_s+\theta_s)\cos \beta_s+\sin \beta_s\sin (\beta_s+\theta_s)}{4\pi(1-\cos \beta_s)}\frac{|p_s|}{2(N+1)M_k\epsilon_k}\bigg )+o(\epsilon_k)\\
&=4\epsilon_k\sum_{s=1}^N\frac{\cos \theta_s}{1-\cos \beta_s}\frac{|p_s|}{2(N+1)M_k\epsilon_k} +o(\epsilon_k). 
\end{align*}
Similarly
\begin{align*}
0&=\partial_2\tilde w_k(e_1)\\
&=16\pi\epsilon_k\sum_{s=1}^N\bigg (\frac{\cos(\beta_s+\theta_s)\sin \beta_s-\cos \beta_s\sin (\beta_s+\theta_s)}{4\pi(1-\cos \beta_s)}\frac{|p_s|}{2(N+1)M_k\epsilon_k}\bigg )+o(\epsilon_k)\\
&=-4\epsilon_k\sum_{s=1}^N\frac{\sin \theta_s}{1-\cos \beta_s}\frac{|p_s|}{2(N+1)M_k\epsilon_k} +o(\epsilon_k). 
\end{align*}
If we use $a_s$ to denote
\[a_s:=\lim_{k\to \infty}\frac{|p_s^k|}{(1-cos\beta_s)M_k\epsilon_k},\quad s=1,...,N,\]
then we have
\begin{equation}\label{eq-ar-1}
\sum_{s=1}^Na_s\cos \theta_s=0
\end{equation}
\begin{equation}\label{eq-ar-2}
\sum_{s=1}^Na_s\sin \theta_s=0,
\end{equation}
where $a_s\ge 0$ for all $s$. Taking the sum of the squares of (\ref{eq-ar-1}) and (\ref{eq-ar-2}) we obtain 
\[\sum_{s=1}^Na_s^2+\sum_{s<t}2a_sa_t\cos(\theta_s-\theta_t)=0. \]
Since $\sum_sa_s^2\ge 2\sum_{s<t}a_sa_t$, we have
\[\sum_{s<t}2a_sa_t(1+\cos(\theta_s-\theta_t))\le 0.\]
Since each term on the left is obviously non-negative, we know each 
\[a_sa_t(1+\cos (\theta_s-\theta_t))=0,\quad \forall s<t. \] 
If there is only one $a_s>0$, it is easy to see that (\ref{eq-ar-1}) and (\ref{eq-ar-2}) cannot both hold. If there are three $a_t's>0$, it is also elementary to see this is not possible: say $a_1,a_2,a_3>0$, then they have to be equal. Then we see that we must have 
\[\theta_1-\theta_2=\pm \pi,\quad \theta_2-\theta_3=\pm \pi,\quad \theta_1-\theta_3=\pm \pi .\] 
Obviously these three equations cannot hold at the same time. So the only situation left is there are exactly two $a_t's$ positive. All other $a_ts$ are zero. Since $p_0=0$, this means there are exactly two $p_{s_1}^k$, $p_{s_2}^k$ such that 
\begin{equation}\label{p-ex-1}\lim_{k\to \infty}\frac{p_{s_1}^k}{\epsilon_kM_k}=-\lim_{k\to \infty}\frac{p_{s_2}^k}{\epsilon_k M_k}\neq 0,\quad \lim_{k\to \infty}\frac{p_t^k}{\epsilon_kM_k}=0,\quad \forall t\neq s_1,s_2.
\end{equation}

If we apply the same argument to $\tilde w_l^k$. Then from $\nabla \tilde w_l^k(Q_l^k)=0$ we would get exactly $p_{l_1}^k$ and $p_{l_2}^k$ different from $p_l^k$ and 
\[\lim_{k\to \infty} \frac{p_{l_1}^k-p_l^k}{\epsilon_k M_k}=-\lim_{k\to \infty}\frac{p_{l_2}^k-p_l^k}{\epsilon_k M_k}\neq 0,\quad \lim_{k\to \infty}\frac{p_t^k-p_l^k}{\epsilon_k M_k}=0, \forall t\neq l_1,l_2.\]
Then it is easy to see that this is only possible when we have $N=2$ because if $N=1$, we would have just one $a_s\neq 0$, which is not possible based on (\ref{eq-ar-1}) and (\ref{eq-ar-2}). If $N\ge 3$, we have to have $p_t^k$ that satisfies 
\[\lim_{k\to \infty}\frac{|p_t^k|}{\epsilon_k M_k}=0,\quad \mbox{and}\quad \lim_{k\to \infty} \frac{p_t^k-p_{s_1}^k}{\epsilon_kM_k}=0\]
which is a contradiction to (\ref{p-ex-1}). 

Finally we rule out the case $N=2$. In this case we have $p_0^k=0$, 
\begin{equation}\label{p-ex-2}\lim_{k\to \infty}\frac{p_1^k}{\epsilon_kM_k}=-\lim_{k\to \infty}\frac{p_2^k}{\epsilon_kM_k}\neq 0.
\end{equation}
However from $\tilde w_1^k(p_1^k)=0$ we have 
\[\lim_{k\to \infty}\frac{p_2^k-p_1^k}{\epsilon_kM_k}=-\lim_{k\to \infty}\frac{0-p_1^k}{\epsilon_kM_k}\neq 0,\]
which is a contradiction to (\ref{p-ex-2}).
Lemma \ref{small-other} is established. $\Box$ 

\medskip

Proposition \ref{key-w8-8} is an immediate consequence of Lemma \ref{small-other}.  $\Box$.

\medskip

Now we finish the proof of Theorem \ref{vanish-first-h}.

Let $\hat w_k=w_k/\tilde \delta_k$. (Recall that $\tilde \delta_k=\delta_k |\nabla \mathfrak{h}_k(0)|+\delta_k^2$). If $|\nabla \mathfrak{h}_k(0)|/\delta_k\to \infty$, we see that in this case 
$\tilde \delta_k\sim \delta_k|\nabla \mathfrak{h}_k(0)|$. The equation of $\hat w_k$ is
\begin{equation}\label{hat-w}
\Delta \hat w_k+|y|^{2N}e^{\xi_k}\hat w_k=a_k\cdot (e_1-y)|y|^{2N}e^{V_k}+b_ke^{V_k}|y-e_1|^{2}|y|^{2N},
\end{equation}
in $\Omega_k$, where $a_k=\delta_k\nabla \mathfrak{h}_k(0)/\tilde \delta_k$, $b_k=o(1)$. 
By Proposition \ref{key-w8-8}, $|\hat w_k(y)|\le C$. Before we carry out the remaining part of the proof we observe that $\hat w_k$ converges to a harmonic function in $\mathbb R^2$ minus finite singular points. Since $\hat w_k$ is bounded, all these singularities are removable. Thus $\hat w_k$ converges to a constant. Based on the information around $e_1$, we claim this constant is $0$. To do this we use the notation $W_k$ again and use Proposition \ref{key-w8-8} to rewrite the equation for $W_k$.
Let
$$W_k(z)=\hat w_k(e_1+\epsilon_k z), \quad |z|< \delta_0 \epsilon_k^{-1} $$
for $\delta_0>0$ small. Then from Proposition \ref{key-w8-8} we have
\begin{equation}\label{h-exp}
\mathfrak{h}_k(\delta_ky)=\mathfrak{h}_k(\delta_k e_1)+\delta_k \nabla \mathfrak{h}_k(\delta_k e_1)(y-e_1)+O(\delta_k^2)|y-e_1|^2,
\end{equation}
\begin{equation}\label{y-1}
|y|^{2N}=|e_1+\epsilon_k z|^{2N}=1+O(\epsilon_k)|z|,
\end{equation}
\begin{equation}\label{v-radial}
V_k(e_1+\epsilon_k z)+2\log \epsilon_k=U_k(z)+O(\epsilon_k)|z|+O(\epsilon_k^2)(\log (1+|z|))^2
\end{equation}
and
\begin{equation}\label{xi-radial}
\xi_k(e_1+\epsilon_k z)+2\log \epsilon_k=U_k(z)+O(\epsilon_k)(1+|z|).
\end{equation}
Using (\ref{h-exp}),(\ref{y-1}),(\ref{v-radial}) and (\ref{xi-radial}) in (\ref{hat-w}) we write the equation of $W_k$ as
\begin{equation}\label{W-rough}
\Delta W_k+\mathfrak{h}_k(\delta_k e_1)e^{U_k(z)}W_k=-\epsilon_k a_k\cdot ze^{U_k(z)}+E_w, \quad 0<|z|<\delta_0 \epsilon_k^{-1}
\end{equation}
where
\begin{equation}\label{ew-rough-2}
E_w(z)=O(\epsilon_k)(1+|z|)^{-3}, \quad |z|<\delta_0 \epsilon_k^{-1}.
\end{equation}

Since $\hat w_k$ obviously converges to a global harmonic function with removable singularity, we have $\hat w_k\to \bar c$ for some $\bar c\in \mathbb R$. Then we claim that
\begin{equation}\label{bar-c-0} \bar c=0.
\end{equation}

The proof of (\ref{bar-c-0}) is the similar as before: If $\bar c\neq 0$, we use $W_k(z)=\bar c+o(1)$ on $B(0,\delta_0 \epsilon_k^{-1})\setminus B(0, \frac 12\delta_0 \epsilon_k^{-1})$ and consider the projection of $W_k$ on $1$:
 $$g_0(r)=\frac 1{2\pi}\int_{0}^{2\pi}W_k(re^{i\theta})d\theta. $$
If we use $F_0$ to denote the projection to $1$ of the right hand side we have, 
$$g_0''(r)+\frac 1r g_0'(r)+\mathfrak{h}_k(\delta_ke_1)e^{U_k(r)}g_0(r)=F_0,\quad 0<r<\delta_0 \epsilon_k^{-1} $$
where
$$F_0(r)=O(\epsilon_k^{\epsilon})(1+|z|)^{-3}. $$
In addition we also have
$$\lim_{k\to \infty} g_0(\delta_0 \epsilon_k^{-1})=\bar c+o(1). $$
For simplicity we omit $k$ in some notations. By the same argument as in Lemma \ref{w-around-e1},  we have
$$g_0(r)=O(\epsilon_k^{\epsilon})\log (2+r),\quad 0<r<\delta_0 \epsilon_k^{-1}. $$
Thus $\bar c=0$.

\medskip

Based on Lemma \ref{bar-c-0} and standard Harnack inequality for elliptic equations we have
\begin{equation}\label{small-til-w}
\tilde w_k(x)=o(1),\,\,\nabla \tilde w_k(x)=o(1),\,\, x\in B_3\setminus (\cup_{l=1}^N(B(e^{i\beta_l},\delta_0)\setminus B(e^{i\beta_l}, \delta_0/8))).
\end{equation}
Equation (\ref{small-til-w}) is equivalent to $w_k=o(\tilde \delta_k)$ and $\nabla w_k=o(\tilde \delta_k)$ in the same region.

\medskip

In the next step we consider the difference between two Pohozaev identities.
 For $\Omega_{0,k}=B(e_1^k,r)$ ($r>0$ small) we consider the Pohozaev identity around $e_1$. For $v_k$ we have
\begin{align}\label{pi-vk}
\int_{\Omega_{0,k}}\partial_{\xi}(|y|^{2N}\mathfrak{h}_k(\delta_ky))e^{v_k}-\int_{\partial \Omega_{0,k}}e^{v_k}|y|^{2N}\mathfrak{h}_k(\delta_ky)(\xi\cdot \nu)\\
=\int_{\partial \Omega_{0,k}}(\partial_{\nu}v_k\partial_{\xi}v_k-\frac 12|\nabla v_k|^2(\xi\cdot \nu))dS. \nonumber
\end{align}
where $\xi$ is an arbitrary unit vector. Correspondingly the Pohozaev identity for $V_k$ is

\begin{align}\label{pi-Vk}
\int_{\Omega_{0,k}}\partial_{\xi}(|y|^{2N}\mathfrak{h}_k(\delta_ke_1))e^{V_k}-\int_{\partial \Omega_{0,k}}e^{V_k}|y|^{2N}\mathfrak{h}_k(\delta_k e_1)(\xi\cdot \nu)\\
=\int_{\partial \Omega_{0,k}}(\partial_{\nu}V_k\partial_{\xi}V_k-\frac 12|\nabla V_k|^2(\xi\cdot \nu))dS. \nonumber
\end{align}
Using $w_k=v_k-V_k$ and $|w_k(y)|=o(\tilde \delta_k)$ on $\partial \Omega_{0,k}$ we have

\[\int_{\partial \Omega_{0,k}}(\partial_{\nu}v_k\partial_{\xi}v_k-\frac 12|\nabla v_k|^2(\xi\cdot \nu))
-\int_{\partial \Omega_{0,k}}(\partial_{\nu}V_k\partial_{\xi}V_k-\frac 12|\nabla V_k|^2(\xi\cdot \nu))
=o(\tilde \delta_k).\]
The difference between the second term of (\ref{pi-vk}) and the second term of (\ref{pi-Vk}) is minor: If we use the expansion of $v_k=V_k+w_k$ and that of $\mathfrak{h}_k(\delta_ky)$ around $e_1$, it is easy to obtain
$$\int_{\partial \Omega_{0,k}}e^{v_k}|y|^{2N}\mathfrak{h}_k(\delta_ky)(\xi\cdot \nu)-\int_{\partial \Omega_{0,k}}e^{V_k}|y|^{2N}\mathfrak{h}_k(\delta_ke_1)(\xi\cdot \nu)=o(\tilde \delta_k). $$
To evaluate the first term,  we use the following updated estimate of $w_k$:
\[ |w_k(e_1+\epsilon_kz)|\le C\tilde \delta_k\epsilon_k1(1+|z|),\quad |z|\le r\epsilon_k^{-1}.\]
Using this expression and 
\begin{align}\label{imp-1}
&\partial_{\xi}(|y|^{2N}\mathfrak{h}_k(\delta_ky))e^{v_k}\\
=&\partial_{\xi}(|y|^{2N}\mathfrak{h}_k(\delta_ke_1)+|y|^{2N}\delta_k\nabla \mathfrak{h}_k(\delta_ke_1)(y-e_1)+O(\delta_k^2))e^{V_k}(1+w_k+O(\delta_k^2))\nonumber\\
=&\partial_{\xi}(|y|^{2N})\mathfrak{h}_k(\delta_k e_1)e^{V_k}+\delta_k\partial_{\xi}(|y|^{2N}\nabla \mathfrak{h}_k(\delta_ke_1)(y-e_1))e^{V_k}\nonumber\\
&+\partial_{\xi}(|y|^{2N}\mathfrak{h}_k(\delta_ke_1))e^{V_k}w_k+O(\delta_k^2)e^{V_k},\nonumber
\end{align}
 we have
\begin{equation}\label{extra-1}
\int_{\Omega_{0,k}}\partial_{\xi}(|y|^{2N})\mathfrak{h}_k(\delta_ke_1)e^{V_k}w_k=o(\tilde \delta_k).
\end{equation}

For the second term on the right hand side of (\ref{imp-1}), we have
\begin{align}\label{imp-2}
&\int_{\Omega_{0,k}}\delta_k\partial_{\xi}(|y|^{2N}\nabla \mathfrak{h}_k(\delta_ke_1)(y-e_1))e^{V_k}\\
=&2N\delta_k\int_{\Omega_{0,k}}y_{\xi}|y|^{2N-2}\nabla \mathfrak{h}_k(\delta_ke_1)(y-e_1)e^{V_k}+
\delta_k\int_{\Omega_{0,k}}|y|^{2N}\partial_{\xi}\mathfrak{h}_k(\delta_ke_1)e^{V_k} \nonumber \\
=&
\delta_k\partial_{\xi}(\log \mathfrak{h}_k)(\delta_ke_1)(8\pi+O(\mu_k\epsilon_k^2))+o(\tilde \delta_k).\nonumber
\end{align}
Using (\ref{extra-1}) and (\ref{imp-2}) in the difference between (\ref{pi-vk}) and (\ref{pi-Vk}), we have
$$\delta_k \partial_{\xi}\mathfrak{h}_k(\delta_ke_1)(1+O(\mu_k\epsilon_k^2))=o(\tilde \delta_k). $$
Thus $|\nabla \mathfrak{h}_k(\delta_ke_1)|=O(\delta_k)$ if $\delta_k=o(\epsilon_k\mu_k^{\frac 12})$. When $\delta_k\ge C\epsilon_k\mu_k^{\frac 12}$ we obtain from \cite{wei-zhang-adv} $|\nabla \mathfrak{h}_k(0)|=O(\delta_k)$. Theorem \ref{vanish-first-h} is established.  $\Box$

\section{ Proof of Theorem \ref{main-thm-1}}

The key estimate is the following more refined estimate of $\nabla \mathfrak{h}_k(\delta_kQ_s^k)$.
\begin{prop}\label{final-prop}
\begin{equation}\label{1-order}
|\nabla \mathfrak{h}_k(\delta_kQ_s^k)|=o(\delta_k),\quad s=0,...,N.
\end{equation}
\end{prop}

\noindent{\bf Proof of Proposition \ref{final-prop}:}

Recall that $V_k$ satisfies
$$\Delta V_k+\mathfrak{h}_k(\delta_ke_1)|y|^{2N}e^{V_k}=0\quad \mbox{in}\quad \mathbb R^2,\quad \int_{\mathbb R^2}|y|^{2N}e^{V_k}<\infty.$$
and 
$$w_k=v_k-V_k,\quad w_k(e_1)=|\nabla w_k(e_1)|=0.$$
The key estimate we establish is 
\begin{equation}\label{best-bound}
|w_k(y)|\le C\delta_k^2,\quad y\in B(0,\tau\epsilon_k^{-1}).
\end{equation}
In order to prove (\ref{best-bound}) we shall consider two cases. Either  $\epsilon_k^{-1}|Q_l^k-e^{i\beta_l}|\to 0$ for all $l$, or there exists $l$ that satisfies $\epsilon_k^{-1}|Q_l^k-e^{i\beta_l}|\ge C$. 
Note that by (\ref{distance}) and Theorem \ref{vanish-first-h}, if $\delta_k=o(\epsilon_k^{1/2})$,
(\ref{location-Q}) holds. If (\ref{location-Q}) is violated, $\delta_k\ge C\epsilon_k^{1/2}$. 

We first prove (\ref{best-bound}) under the assumption 
\begin{equation}\label{key-assump-1}
\epsilon_k^{-1}|Q_l^k-e^{i\beta_l}|\to 0\quad \mbox{ for all}\quad l.
\end{equation}
If (\ref{best-bound}) does not hold, let $M_k=\max|w_k|$ and suppose $M_k/\delta_k^2\to \infty$. Then we set 
$$\hat w_k=w_k/M_k, $$
the equation of $\hat w_k$ is 
$$\Delta \hat w_k(y)+|y|^{2N}\mathfrak{h}_k(\delta_ke_1)e^{\xi_k}\hat w_k(y)=a_k\cdot (e_1-y)|y|^{2N}e^{V_k}+\hat E_1$$
where $a_k=\delta_k\nabla \mathfrak{h}_k(0)/M_k\to 0$ and $\hat E_k=o(1)|y-e_1|^2|y|^{2N}e^{V_k}$. 

Then, as before, we prove that $\hat w_k=o(1)$ outside bubbling disks.
The proof of this part is as before, first we prove $\hat w_k=o(1)$ near $e_1$. Then we use Harnack inequality to pass through the smallness away from singular sources. 

Then we get a contradiction just as the proof of Proposition \ref{vanish-first-h}, the part before the evaluation of the Pohozaev identity, which essentially says that $v_k$ cannot be very similar to a global solution at $N+1$ local maximums. (\ref{best-bound}) is established under the assumption (\ref{key-assump-1}).

Still under (\ref{key-assump-1}) we use the notation $\hat w_k$ and let it be defined as
$$\hat w_k(y)=w_k(y)/\delta_k^2.$$
Then the equation of $\hat w_k$ is 
\begin{align}\label{hat-wk}
&\Delta \hat w_k(y)+|y|^{2N}\mathfrak{h}_k(\delta_ke_1)e^{\xi_k}\hat w_k(y)\\
=&a_k\cdot (e_1-y)|y|^{2N}e^{V_k}+\sum_{|\alpha |=2}\frac{\partial^{\alpha}\mathfrak{h}_k(\delta_ke_1)}{\alpha!}(e_1-y)^{\alpha}|y|^{2N}e^{V_k}\nonumber\\
&+O(\delta_k)|y-e_1|^3|y|^{2N}e^{V_k}.\nonumber
\end{align}
where $|a_k|=O(1)$ by (\ref{vanish-first-h}). $\Delta \hat w_k \to 0$ away from finite points and $\hat w_k$ is bounded, which means the singularities are removable. Thus $\hat w_k\to c$ in $\mathbb R^2$ minus a few points. By looking at the ODE projected to $1$ we see that $c=0$.

Because of the smallness of $w_k$ ($w_k=O(\delta_k^2))$ we can compare the Pohozaev identities of $v_k$ and $V_k$ and obtain $\nabla \mathfrak{h}_k(\delta_ke_1)=o(\delta_k)$ as before. Similarly $\nabla \mathfrak{h}_k(\delta_kQ_l^k)=o(\delta_k)$ can be obtained similarly. 
Proposition \ref{final-prop} is established under the assumption (\ref{key-assump-1}).

\medskip

Now we prove (\ref{best-bound}) if (\ref{key-assump-1}) does not hold: There exists $s$ such that 
\[\epsilon_k^{-1}|Q_s^k-e^{i\beta_s}|\ge C.\]
In this case necessarily we have $\delta_k>C\epsilon_k^{1/2}$ and, based on standard result for Liouville equation (say, Theorem 1.1 of \cite{gluck}), $M_k\ge C$.
In this case we set $V_k$ to be the solution that agrees with $v_k$ at $e_1$ and let $w_k=v_k-V_k$.

 Since $M_k\ge C$, we just focus on $w_k$ itself. Using $ w_k(e_1)=0$ we have 
 \begin{align*}
 & w_k(y)\\
 =&\int_{\Omega_k}(G_k(y,\eta)-G_k(e_1,\eta))|\eta |^{2N}(\mathfrak{h}_k(\delta_k\eta)e^{v_k}-\mathfrak{h}_k(\delta_ke_1)e^{V_k})d\eta\\
=&-\frac{1}{2\pi}\int_{\Omega_k}\log \frac{|y-\eta |}{|e_1-\eta|} |\eta |^{2N}(\mathfrak{h}_k(\delta_k\eta)e^{v_k}-\mathfrak{h}_k(\delta_ke_1)e^{V_k})d\eta +o(\delta_k^2)
 \end{align*}
 Here we recall 
 \[\epsilon_k\le C\max_{l}|Q_l^k-e^{i\beta_l}|\le C\delta_k^2. \]
 When we evaluate the integral in the expression of $\tilde w_k$ above, we see immediately that we only need to evaluate each $B(Q_l^k,\tau_2)$ for some $\tau_2>0$ small. Thus we have
 \begin{align}\label{far-1}
 &\tilde w_k(y)=\int_{\Omega_k}L(y,\eta)A(\eta)d\eta+o(\delta_k^2).\\
 &=\sum_{l=1}^N\int_{B(Q_l^k,\tau_3)}(L(y,\eta)-L(y,Q_l^k))A(\eta)d\eta+L(y,Q_l^k)A_l^k+o(\delta_k^2). \nonumber 
 \end{align}
 where 
 \[L(y,\eta)=-\frac{1}{2\pi}\log \frac{|y-\eta |}{|e_1-\eta |},\] 
 \[A(\eta)=|\eta |^{2N}(\mathfrak{h}_k(\delta_k\eta)e^{v_k}-\mathfrak{h}_k(\delta_ke_1)e^{V_k}).\]
 and 
\[A_l^k=\int_{B(Q_l^k,\tau_3)}A(\eta)d\eta.\]
 To evaluate $w_k(y)$ we first observe that 
\begin{equation}\label{Alk}\int_{B(Q_l^k,\tau)}(L(y,\eta)-L(y,Q_l^k))A(\eta)=o(\epsilon_k)=o(\delta_k^2)
\end{equation}
 The reason is $y-Q_l^k$ becomes $\epsilon_kz$ after the following change of variable: $y=Q_l^k+\epsilon_kz$. Then we look at (\ref{late-1}) and use the cancellation in $\phi_1$. Then it is not hard to see that the result is $o(\epsilon_k)$. For $A_l^k$ we use the standard result for Liouville equation (see \cite{chenlin1,zhangcmp,gluck}) to have 
 \begin{equation}\label{Alk-1}
 A_l^k=O(\epsilon_k^2\log \frac{1}{\epsilon_k})=o(\delta_k^2),\quad l=1,..,N.
 \end{equation}

As a consequence $\tilde w_k=o(\delta_k)$ away from the singular sources. Around $e_1$, 
\[|\tilde w_k(e_1+\epsilon_kz)|\le o(\delta_k\epsilon_k))(1+|y|),\quad |y|\le \tau_1\epsilon_k^{-1}.\]
By comparing the Pohozaev identities of $V_k$ and $v_k$ around $B(e_1,\tau_1)$, we have $|\nabla \mathfrak{h}_k(\delta_ke_1)|=o(\delta_k)$. Applying the same argument around each $Q_s^k$ we have
\[|\nabla \mathfrak{h}_k(\delta_kQ_s^k)|=o(\delta_k),\quad s=0,..,N.\] 
Proposition \ref{final-prop} is established in all cases. $\Box$

\medskip 

Theorem \ref{main-thm-1} immediately follows from Proposition \ref{final-prop}: For $N=1$, $\partial_{1}\mathfrak{h}_k(\delta_1 e_1)$ and $\partial_1\mathfrak{h}_k(\delta_k(-e_1))$ are both $o(\delta_k)$ implies $\partial_{11}\mathfrak{h}_k(0)=o(1)$. The fact that $\partial_2\mathfrak{h}_k(\delta_k e_1)$ and $\partial_2 \mathfrak{h}_k(\delta_k(-e_1))$ being both $o(\delta_k)$ implies that $\partial_{12}\mathfrak{h}_k(0)=o(1)$. Finally by $\Delta \mathfrak{h}_k(0)=o(1)$ proved in \cite{wei-zhang-jems} we have $\partial_{22}\mathfrak{h}_k(0)=o(1)$. When $N\ge 2$ we evaluate 
\[\nabla \mathfrak{h}_k(\delta_kQ_l^k)-\nabla\mathfrak{h}_k(\delta_ke_1)=o(\delta_k)\]
for $l\neq 0$. Then it is easy to prove $\partial_{ij}\mathfrak{h}_k(0)=o(1)$ for all $i,j=1,2.$
Theorem \ref{main-thm-1} is established. $\Box$

\section{Higher order vanishing theorem.}

 The key estimate in this section is 
\begin{prop}\label{1st-d} For $N\ge 2$ we have 
\begin{equation}\label{crucial-third}
\nabla \mathfrak{h}_k(\delta_k Q_l^k)=o(\delta_k^2),\quad \forall l=0,1,...,N
\end{equation}
\end{prop}

\noindent{\bf Proof of Proposition \ref{1st-d}:} 

 Here we consider two cases: 
\begin{enumerate}
    \item $\max_{l} \epsilon_k^{-1}|Q_l^k-e^{i\beta_l}|\ge C$. Note that in this case based on (\ref{distance}) and (\ref{1-order}) we have $\delta_k\ge C\epsilon_k^{\frac 12}$.
    \item $\epsilon_k^{-1}|Q_l^k-e^{i\beta_l}|=o(1)$ for all $l$. One sufficient condition based on (\ref{distance}) is $\delta_k\le C\epsilon_k^{1/2}$. 
\end{enumerate}
For the first case we prove (\ref{crucial-third}).  Observe that in this this case we have $M_k\ge C$ (so we only work on $w_k$).  Using the argument as in the proof of Theorem \ref{main-thm-1}  we have 
\[w_k(y)=\sum_{l=1}^NL_1(y,\eta)A_l^k+o(\delta_k^3),\quad y\in B_5\setminus \{B(Q_k^l,\tau_2)\cup....\cup B(Q_N^k,\tau_2)).\]
where $A_l^k$ is defined as in (\ref{Alk}) and 
\[L_1(y,\eta)=G_k(y,\eta)-G_k(e_1,\eta),\]
Then using (\ref{Alk-1}) we have
\[w_k(y)=O(\epsilon_k^2\log \frac{1}{\epsilon_k})=o(\delta_k^3), \quad y\in B_5\setminus (\cup_{l=1}^NB(Q_l^k,\tau_2)).\]
Then the proof of $|\nabla \mathfrak{h}_k(\delta_ke_1)|=o(\delta_k^2)$ is carried out as before. When it comes to $l\neq 0$, we can replace $V_k$ by $V_l^k$ and we still have $v_k-V_l^k=o(\delta_k^3)$ in the neighborhood of $\partial B(Q_l^k,\tau_1)$. Thus by evaluating the Pohozaev identities we obtain (\ref{crucial-third}). 

Now under the second possibility, we now prove (\ref{crucial-third}).

The method of approximating blowup solutions using different global solutions at each local maximum can be used. 
We set $w_k=v_k-V_k$ with common local maximum at $e_1$. Then we set 
 \[M_k:=\max_{x\in \bar \Omega_k}|w_k(x)|\]
 Our goal is to show that 
\begin{equation}\label{for-third}
M_k\le C(\delta_k\sum_{l=0}^N|\nabla \mathfrak{h}_k(\delta_k Q_l^k)|+\delta_k^3)
\end{equation}
Here we note that we are using a different argument as in the previous section. 
 Based on what we have done for the Hession vanishing theorem we have already known $M_k=O(\delta_k^2)$.

By way of contradiction we assume that 
\[M_k/(\delta_k\sum_{l=0}^N|\nabla \mathfrak{h}_k(\delta_kQ_l^k)|+\delta_k^3)\to \infty.\]
Then we let $\hat w_k=w_k/M_k$ with the obvious implication $|\hat w_k|\le 1$. Then the equation of $\hat w_k$ is
\begin{equation}\label{3rd-w}
\Delta \hat w_k+\mathfrak{h}_k(\delta_ky)|y|^{2N}e^{\xi_k}\hat w_k=
\frac{\mathfrak{h}_k(\delta_ke_1)-\mathfrak{h}_k(\delta_ky)}{M_k}|y|^{2N}e^{V_k},
\end{equation}
in $\Omega_k:=B(0,\tau \delta_k^{-1})$.
The estimate of $\xi_k$ is as before
\[e^{\xi_k(x)}=e^{V_k}(1+\frac 12 w_k+O(w_k^2))\]
Let 
\[a_l^k=(\mathfrak{h}_k(\delta_ke_1)-\mathfrak{h}_k(\delta_k e^{i\beta_l}))/(\mathfrak{h}_k(\delta_k e^{i\beta_l})M_k).\]
Since 
\[\delta_k\nabla \mathfrak{h}_k(\delta_kQ_l^k)/M_k=o(1), \,\, \forall i, \]
and $\delta_k^3/M_k\to 0$,
the proof of Lemma \ref{lemma-wei} below gives
$\delta_k^2\partial_{ij}\mathfrak{h}_k(0)/M_k=O(1)$ for all $N\ge 2$.

Consequently all $a_l$ are bounded and we can carry out the Fourier analysis around $e_1$, we now observe that all the coefficients for $\tilde w_k$ in the expansion are $O(1)$. 
Then (\ref{for-third}) can be deduced by the same argument as in the proof of Theorem \ref{main-thm-1}. Consequently it is easy to prove that $w_k=o(\delta_k^3)+o(\delta_k\sum_{l}|\nabla \mathfrak{h}_k(\delta_kQ_l)|)$ away from the bubbling disks. 
Since $\mathfrak{h}_k\in C^{N+2}$, obviously the equation above yields
\begin{equation}\label{for-2nd-1}
\partial_{\xi}\mathfrak{h}_k(\delta_k Q_l^k)=O(\delta_k\epsilon_k\sum_l|\nabla \mathfrak{h}_k(\delta_kQ_l^k)|)+o(\delta_k^2).
\end{equation}
Thus  (\ref{crucial-third}) is obtained and  
Proposition \ref{1st-d} is established.. $\Box$

\medskip

Now we complete the proof of Theorem \ref{main-thm-2}.

From $|\nabla \mathfrak{h}_k(\delta_k Q_l^k)|=o(\delta_k^2)$ we have
\[|Q_l^k-e^{i\beta_l}|\le s_k \delta_k^3+C\mu_k e^{-\mu_k}\]
where $s_k\to 0+$. So if 
\[\delta_k\le s_k^{-\frac 16}\epsilon_k^{1/3}, \]
we have 
\[|Q_l^k-e^{i\beta_l}|\epsilon_k^{-1}=o(1).\]
On the other hand if 
\[|Q_l^k-e^{i\beta_l}|\ge C\epsilon_k\]
for at least one $l$, we have 
\[\delta_k\ge C s_k^{-\frac 16}\epsilon_k^{1/3}.\]
In this case $\delta_k^5/(\mu_ke^{-\mu_k})\to \infty$. There we can use the same argument to prove  
\begin{equation}\label{unif-bound-1}
\max_{x\in \Omega_k}|w_k|\le C(\delta_k^5+\delta_k\sum_{l=0}^N|\nabla \mathfrak{h}_k(\delta_kQ_l^k)|).
\end{equation}
In fact in the proof if we set 
$M_k=\max_{\Omega_k} |w_k|$, then the function $\hat w_k=w_k/M_k$ has coefficient functions satisfying
$\delta_k\nabla \mathfrak{h}_k(\delta_kQ_l^k)/M_k=o(1)$ and $\delta_k^5/M_k=o(1)$.  Then by the proof of Lemma \ref{lemma-wei} below, for $N\ge 8$, we have 
\[\frac{\nabla^{\alpha}\mathfrak{h}_k(0)\delta_k^{|\alpha |}}{M_k}=o(1),\quad \mbox{for}\quad |\alpha |=2,3,4.\]
Thus by looking at the expansion of $\hat w_k$ around $e_1$ we can still obtain the smallness of $\hat w_k$ as before. Then using the same argument from evaluating Pohozaev identities we obtain 
\[|\nabla \mathfrak{h}_k(\delta_kQ_l^k)|=o(\delta_k^4),\quad l=0,...,N.\]
By Lemma \ref{lemma-wei} below we have, for $N\ge 8$,
\begin{equation}\label{step-1}
|\nabla^{\alpha}\mathfrak{h}_k(\delta_kQ_l^k)|=o(\delta_k^{5-|\alpha |}),\quad 1\le |\alpha |\le 5.
\end{equation}

For $N\ge 8$ with the new rate of $|\nabla \mathfrak{h}_k(\delta_k Q_l^k)|$, we now have 
\[|Q_l^k-e^{i\beta_l}|=o(\delta_k^5),\quad l=0,...,N.\]
Using the same reasoning, we put the analysis in two cases: Either $\delta_k\le S_k\epsilon_k^{\frac 15}$ for some $S_k\to \infty$, or the compliment. As a result, we prove 
\[\max_{x\in \bar \Omega_k}|w_k|\le C(\delta_k^9+\delta_k\sum_l |\nabla \mathfrak{h}_k(\delta_kQ_l^k)|),\]
which further gives $|\nabla \mathfrak{h}_k(\delta_ke_1)|\le o(\delta_k^8)$ and 
\[|\nabla \mathfrak{h}_k(\delta_k Q_l^k)|=o(\delta_k^8),\quad l=0,...,N.\]
Consequently by Lemma \ref{lemma-wei}, for $N\ge 16$,
we have 
\begin{equation}\label{more-van-1}
|\nabla^{\alpha}\mathfrak{h}_k(0)|=o(\delta_k^{9-|\alpha |}),\quad 1\le |\alpha |\le 9.
\end{equation}
In general for $N\ge 2^{M+1}$ we have 
\[|\nabla^{\alpha} \mathfrak{h}_k(0)|=O(\delta_k^{2^M+1-|\alpha|}), \quad 1\le |\alpha|\le 2^M+1.\]

Finally we present this useful lemma:

\begin{lem}\label{lemma-wei}
Suppose \[|\nabla h_k(\delta_k  e^{i\beta_l})|=o(\delta_k^{N_1}),\quad l=0,...,N, \quad \beta_l=\frac{2\pi l}{N+1},\]
where $N_1$ is a positive integer. Then for $N\ge 2N_1$, we have
\[|\nabla^{\alpha}  h_k(0) | =o(\delta_k^{N_1+1-|\alpha |}), \quad \forall 1\le |\alpha |\le N_1+1. \] 
\end{lem}

\noindent{\bf Proof of Lemma \ref{lemma-wei}:}

Consider a real valued function $f(x,y)$  being approximated  by a polynomial of degree $ N_1$:
\[f(x,y)=\sum_{i+j\le N_1}a_{ij}z^i\bar z^j+o(r^{N_1}),\]
where we use $z=x+iy$ and $\bar z=x-iy$. If we use $z$ and $\bar z$ as two variables we have 
\begin{equation}\label{wei-1}f(z,\bar z)=\sum_{i=0}^{N_1}a_{ii}r^{2i}+\sum_{i+j\le N_1,i\neq j}a_{ij}r^{i+j}(\frac{z}{r})^i(\frac{\bar z}r)^j+o(r^{N_1}),
\end{equation}
where $r=|z|$.
Since $f$ is real valued we have $a_{ij}=\bar a_{ji}$. 
Let $z_0= e^{\frac{2i\pi}{N+1}}$, suppose
$$ f(\delta_k z_0^m) = o(\delta_k^{N_1}), m=0,1,..., N_1. $$
Then we claim that for $N\ge 2N_1$,
\[f(x,y)=o(\delta_k^{N_1}),\quad \forall |(x,y)|=\delta_k.\]
To prove this we write $f(\delta_k z_0^m)$ as
\begin{align*} f(\delta_k z_0^m)&
= \sum_{i+j\leq N_1} a_{ij} \delta_k^{i+j} z_0^{m(i-j)} + o(\delta_k^{N_1})\\
&= \sum_{i=0}^{N_1} a_{ii} \delta_k^{2i}+ \sum_{ 1\leq i+j\le N_1, i\neq j} a_{ij} \delta_k^{i+j} z_0^{m(i-j)} + o(\delta_k^{N_1})
\end{align*}

Let $ N\ge 2N_1$. For any $c_m, m=0,1,..., 2N_1$ we have
\begin{align*} 
&o(\delta_k^{2N_1})=\sum_{m=0}^{2N_1} c_m f(\delta_k z_0^m)\\
=& \sum_{i=0}^{2N_1} a_{ii} \delta_k^{2i} \sum_{m=0}^{2N_1} c_m + \sum_{ 1\leq i+j\le N_1, i\neq j} a_{ij} \delta_k^{i+j} \sum_{m=0}^{2N_1}  c_m z_0^{(i-j)m}  + o(\delta_k^{N_1})
\end{align*}
Then for any given $b_{-N_1},...,b_0,b_1,..,b_{N_1}$, we consider the system
\[\sum_{m=0}^{2N_1}c_m=b_0,\]
\[\sum_{k=0}^{2N_1}c_k(z_0^{i-j})^k=b_{i-j},\quad i-j=-N_1,...,-1,1,...,N_1.\]
Then we see that the coefficient matrix for $c_0,...,c_{2N_1}$ is a Vandermonde matrix because for $N\ge 2N_1$, 
\[z_0^{-N_1},...,z_0^{-1},1,z_0,...,z_0^{N_1}\]
are all distinct. Thus by choosing one $b_l=1$ all others equal to $0$  we have 
\[\sum_{i=0}^{2N_1}a_{ii}\delta_k^{2k}=o(\delta_k^{N_1})\]
and 
\[\sum_{j-i=l}a_{ij}\delta_k^{i+j}=o(\delta_k^{N_1}), \quad l=-N_1,...,N_1.\]

Putting these two estimates together we obtain from (\ref{wei-1}) that 
\[f(z,\bar z)=o(\delta_k^{N_1}),\quad |z|=\delta_k.\]
Therefore we have 
$$ \partial^m_z \partial_{\bar{z}}^n f (0)= o(\delta_k^{N_1-m-n}), m+n \leq N_1 $$

Now we replace $f$ by two functions:
$\partial_x\mathfrak{h}_k=(\partial_z+\partial_{\bar z})\mathfrak{h}_k$ and $\partial_y\mathfrak{h}_k=i(\partial_{z}-\partial_{\bar z})\mathfrak{h}_k$. Using the result above we have 
\[\partial^m_z\partial^n_{\bar z}(\partial_z+\partial_{\bar z})\mathfrak{h}_k(0)=o(\delta_k^{N_1-(m+n)}),\quad \forall m+n\le N_1. \]
and 
\[\partial^m_z\partial^n_{\bar z}(\partial_z-\partial_{\bar z})\mathfrak{h}_k(0)=o(\delta_k^{N_1-(m+n)}),\quad \forall m+n\le N_1. \]
Putting these equations together we have 
\[\nabla^{\alpha}\mathfrak{h}_k(0)=o(\delta_k^{N_1+1-|\alpha |}),\quad \forall 1\le |\alpha |\le N_1+1,\]
for $N\ge 2N_1$. Lemma \ref{lemma-wei} is established. $\Box$

\medskip

Consequently Theorem \ref{main-thm-2} is established. $\Box$

\section{Dirichlet Problem}

In this section we address the following Dirichlet problem:

\begin{equation}\label{eq000}
  \left\{
      \begin{aligned}&- \Delta v_k = \la_k |x|^{2N}V(x)e^{v_k}
      &  \hbox{ in }&  \Omega,\\
    &  \ v_k=0 &  \hbox{ on }& \partial \Omega.
  \end{aligned}
    \right. \end{equation}
    where $N\geq 1\in \mathbb N$. For a sequence of blowup solutions $v_k$, it is standard
to assume a uniform bound on the total integration (see \cite{batta}): there
exists $C > 0$ independent of $k$ such that

\beq\label{unifbound}\la_k\into |x|^{2N}V(x)e^{v_k}dx\leq C.\eeq

We also suppose that $v_k$ admits the origin as its only blow-up point in $B_1$, in other words
\beq\label{eqq000}\max_{\Omega} v_k(x)\to +\infty,  \eeq
and 
 for any compact set  $K\subset\Omega\setminus\{0\}$ there exists a constant $C(K)$ (depending on $K$) such that 
\beq\label{eqq0000}\max_{K}v_k\leq C(K).\eeq 

The Dirichlet problem when $N=1$ and $\Omega$ is the unit ball $B_1=\{x\in\R^2\,|\, |x|<1\}$ can be described as 
\begin{equation}\label{eq0}
  \left\{
      \begin{aligned}&- \Delta v_k = \la_k V(x) |x|^{2} e^{v_k}&  \hbox{ in }&  B_1,\\
    &  \ v_k=0 &  \hbox{ on }& \partial B_1,\\ &\lambda_k \int_{B_1} |x|^{2}e^{u_k} dx \leq C
  \end{aligned}
    \right. \end{equation}
and we suppose that $v_k$ un admits the origin as its only blow
up point in $B_1$ \beq\label{eqq0}\max_{B_1}v_k\to +\infty, \quad \sup_{\e\leq |x|\leq 1}v_k<C(\e)\quad \forall \e>0.\eeq
In addition, we postulate the usual assumption on $V$$V$\beq\label{eqq1} V\in C^1(\overline B_1), \quad V\in C^2(U),\quad \min_{\overline B_1}V>0,\quad \nabla V(0)=0\eeq
where $U$ is a neighborhood of the origin.

We point out that the hypothesis on the vanishing of the first derivatives of the potential $V$ is a necessary condition for the non-simple blowup solutions (see \cite{wei-zhang-plms}).

The main result in this section is we use two new Pohozaev identities to prove the vanishing property of $D^2V$:

\begin{thm}\label{th2} Assume that the sequence  $v_k$   satisfies \eqref{eq0}-\eqref{eqq0}  and has the non-simple blow-up property. Then, if $V$ verifies \eqref{eqq1}, the following holds
\[D^{\alpha}V(0)=0\quad \forall |\alpha |=2.\]
\end{thm}

\

 The proof relies on the derivation of two new Pohozaev identities.  Moreover, the method fails for $N\geq 2$ since in the expansion of the Pohozaev identity we no longer catch information on the Hessian matrix of $V$ (see Remark \ref{fails} for technical details).

\

\subsection{Pohozaev-type Identities}
In this subsection we prove Pohozaev-type identities for 
solutions of \eqref{eq0}. We then exploit such integral identities
to prove conditions on the potential $V$ 
 for the existence of associated families of blowing up solutions provided by Theorem \ref{th2}.
 
First we introduce complex notations and identify $x=(x_1,x_2)\in \R^2$  with $z= x_1+{\rm i}x_2\in\C$  and we denote by $x^N$ the $N$-power of the complex number $x$;
then we fix some algebraic identities in the following lemma. The proof is a direct computation. 

\begin{lem}\label{lemma0} Let $N\in\N$. Then for any $x\in\R^2\setminus\{0\}$ 
the following identities hold:
$$2x_1\Re (x^{N})= \Re(x^{N+1})+|x|^2\Re(x^{N-1}),\quad 2x_2\Re(x^{N})=\Im(x^{N+1})-|x|^2\Im (x^{N-1}),$$ $$ 2x_1\Im(x^{N})=\Im(x^{N+1})+|x|^2\Im(x^{N-1}),\quad 2x_2\Im (x^{N})=|x|^2\Re(x^{N-1})-\Re(x^{N+1}),$$
$$\frac{\partial}{\partial x_1}\big(\Re(x^{N})\big)=\frac{\partial}{\partial x_2}\big(\Im(x^{N})\big)=N\,\Re(x^{N-1}),$$ $$ \frac{\partial}{\partial x_1}\big(\Im(x^{N})\big)=- \frac{\partial}{\partial x_2}\big(\Re(x^{N})\big)=N\,\Im(x^{N-1}),$$ 
$$ \frac{\partial }{\partial x_1}\bigg(\frac{\Re(x^{N})}{|x|^{2N}}\bigg)=-\frac{\partial }{\partial x_2}\bigg(\frac{\Im(x^{N})}{|x|^{2N}}\bigg)=-N\frac{\Re(x^{N+1})}{|x|^{2{(N+1)}}}, $$ $$ \frac{\partial }{\partial x_1}\bigg(\frac{\Im(x^{N})}{|x|^{2N}}\bigg)=\frac{\partial }{\partial x_2}\bigg(\frac{\Re(x^{N})}{|x|^{2N}}\bigg)=-N\frac{\Im(x^{N+1})}{|x|^{2(N+1)}}.$$

\end{lem}

\

We proceed to provide   the first Pohozaev indentity. 
 In the following $G(x, y)$ is the Green's function of $-\Delta$ over $\Omega$ under Dirichlet boundary conditions and $H(x,y)$ denotes its regular part:
$$H(x,y):=G(x,y)-\frac{1}{2\pi}\log\frac{1}{|x-y|}.$$

\begin{prop}\label{propo1}
Let $N\in\N$. Assume that $\Omega$ is a smooth and bounded planar domain such that $0\in\Omega$ and the potential $V>0$ belongs to the class $C^1(\overline\Omega)$. Then any solutions  $v\in C(\overline\Omega)$ of the boundary value problem \beq\label{eq1}\left\{\begin{aligned}& -\Delta v=\la |x|^{2N} V(x) e^v&\hbox{ in } &\Omega\\ &v=0& \hbox{ on } &\partial \Omega\end{aligned}\right.\eeq
satisfies the following Pohozaev identity:
$$\begin{aligned}&\frac12\int_{\partial \Omega}\bigg|\frac{\partial v}{\partial \nu}\bigg|^2  \nu_i dx-\int_{\partial \Omega}\frac{\partial v}{\partial \nu}\frac{\partial v}{\partial x_i} dx-4N\pi \la\int_{ \Omega}\frac{\partial H}{\partial x_i}(0,y) |y|^{2N}V(y) e^v dy\\ &\;\;\;\;-\la\int_{\partial \Omega} |x|^{2N}V(x) \nu_idx
\\ &=-4N\pi\frac{\partial v}{\partial x_i}(0) -\la\into |x|^{2N} \frac{\partial V}{\partial x_i}(x) e^vdx\end{aligned}$$
for $i=1,2.$
\end{prop}

\begin{proof}
Let us multiply on both sides of the equation in \eqref{eq1} by $\frac{\partial v}{\partial x_i}$  (for $i=1,2$) and integrating on $\Omega$ we get
\beq\label{nirv}\into  \nabla v\cdot \nabla \bigg(\frac{\partial v}{\partial x_i}\bigg) dx-\int_{\partial \Omega}\frac{\partial v}{\partial \nu}\frac{\partial v}{\partial x_i} dx=\into\la |x|^{2N} V(x) e^v\frac{\partial v}{\partial x_i} dx.\eeq
Taking into account that $$\nabla v\cdot \nabla \bigg(\frac{\partial v}{\partial x_i}\bigg) =\frac12\frac{\partial }{\partial x_i}\Big(|\nabla v|^2\Big), \qquad e^v\frac{\partial v}{\partial x_i} =\frac{\partial e^v}{\partial x_i} ,$$ 
by the divergence theorem \eqref{nirv} becomes
\beq\label{magl}\begin{aligned}&\frac12\int_{\partial \Omega}|\nabla v|^2  \nu_i dx-\int_{\partial \Omega}\frac{\partial v}{\partial \nu}\frac{\partial v}{\partial x_i} dx
\\ &=-2N\la\into |x|^{2N-2}x_i V e^vdx-\la\into |x|^{2N} \frac{\partial V}{\partial x_i} e^vdx+\la\int_{\partial \Omega} |x|^{2N}V(x) \nu_i dx\end{aligned}\eeq
where we have used the homogeneous boundary condition $v=0$ on $\partial \Omega$. 
By Poisson representation formula we have for $i=1,2$ and $x\in\Omega$
$$\begin{aligned}\frac{\partial v}{\partial x_i}(x)&=\la\into \frac{\partial G}{\partial x_i}(x-y)|y|^{2N}V(y)e^{v_k}(y)dy\\ &=\la\into \bigg(\frac{y_i-x_i}{2\pi|x_i-y_i|^2}+\frac{\partial H}{\partial x_i}(x,y)\bigg)|y|^{2N}V(y)e^{v_k(y)}dy\end{aligned}$$
where $G$ and $H$ are the Green's function and its regular part as defined above. So we deduce
$$\begin{aligned}\la\into y_i|y|^{2N-2}V(y)e^v dy=2\pi \frac{\partial v}{\partial x_i}(0)-2\pi\la\into\frac{\partial H}{\partial x_i}(0,y)|y|^{2N}V(y)e^{v_k(y)}dy.\end{aligned}$$ Combining the last identity with \eqref{magl}
 we get the thesis. 
\end{proof}

\

\begin{prop}\label{propo2} Let $N\in\N$. Assume that $\Omega$ is a smooth and bounded planar domain such that $0\in\Omega$ and  the potential $V>0$ belongs to the class $C^1(\overline\Omega)$. Then any solutions  $v\in C^1(\overline\Omega)$ of the boundary value problem \eqref{eq1} 
satisfies the following Pohozaev identity:
$$\begin{aligned}&\la\int_{\Omega}  e^{v}\bigg(\frac{\partial V(x)}{\partial x_1}\Re(x^{N})-\frac{\partial V(x)}{\partial x_2}\Im(x^{N})\bigg) dx
\\ &\;\;\;\;-\la\int_{\partial \Omega}  V(x)e^v \Big(\Re(x^{N})\nu_1-\Im(x^{N})\nu_2\Big)d\sigma\\ &=\frac12\int_{\partial \Omega}\bigg|\frac{\partial v}{\partial \nu}\bigg|^2 \bigg(\frac{\Re(x^{N})}{|x|^{2N}}\nu_1-\frac{\Im(x^{N})}{|x|^{2N}}\nu_2\bigg) d\sigma
\\ &\;\;\;\;-\frac{1}{2\e^{2N-1}}\int_{\partial B_\e}\bigg(\bigg|\frac{\partial v}{\partial x_1}\bigg|^2-\bigg|\frac{\partial v}{\partial x_2}\bigg|^2\bigg)\Re(x^{N-1})d\sigma\\ &\;\;\;\;+
\frac{1}{\e^{2N-1}}\int_{\partial B_\e}\frac{\partial v}{\partial x_1}\frac{\partial v}{\partial x_2}\Im(x^{N-1})d\sigma+o(1)
\end{aligned}$$ as $\e\to 0^+$, 
where $\nu$ stands for the unit outward normal. 
\end{prop}

\begin{proof} Let us multiply  both sides of the equation in \eqref{eq1} by $$\frac{\partial v}{\partial x_1}\frac{\Re(x^{N})}{|x|^{2N}}-\frac{\partial v}{\partial x_2}\frac{\Im(x^{N})}{|x|^{2N}};$$ using that $v$ belongs to $C^1(\overline{\Omega})$ by standard regularity theory   and integrating on $\Omega\setminus B_\e$ for $\e>0$ sufficiently small we get 
\beq\label{leftright}\begin{aligned}\int_{\Omega\setminus B_\e} (-\Delta v)&\bigg(\frac{\partial v}{\partial x_1}\frac{\Re(x^{N})}{|x|^{2N}}-\frac{\partial v}{\partial x_2}\frac{\Im(x^{N})}{|x|^{2N}}\bigg)dx\\ &=\la \int_{\Omega\setminus B_\e}V(x) e^{v}\bigg(\frac{\partial v}{\partial x_1}\Re(x^{N})-\frac{\partial v}{\partial x_2}\Im(x^{N})\bigg)dx.\end{aligned}\eeq
By applying Gauss Green formula, we have
$$\begin{aligned}&\int_{\Omega\setminus B_\e} (-\Delta v)\bigg(\frac{\partial v}{\partial x_1}\frac{\Re(x^{N})}{|x|^{2N}}-\frac{\partial v}{\partial x_2}\frac{\Im(x^{N})}{|x|^{2N}}\bigg)dx\\ &=
\int_{\Omega\setminus B_\e} \nabla v\cdot \nabla\bigg(\frac{\partial v}{\partial x_1}\frac{\Re(x^{N})}{|x|^{2N}}-\frac{\partial v_k}{\partial x_2}\frac{\Im(x^{N})}{|x|^{2N}}\bigg)dx\\ &\;\;\;\;-\int_{\partial \Omega\cup\partial B_\e }\frac{\partial v}{\partial \nu}\bigg(\frac{\partial v}{\partial x_1}\frac{\Re(x^{N})}{|x|^{2N}}-\frac{\partial v}{\partial x_2}\frac{\Im(x^{N})}{|x|^{2N}}\bigg)d\sigma .
\end{aligned} $$
Using Lemma \ref{lemma0} we derive
$$\begin{aligned}&\nabla v\cdot \nabla \bigg(\frac{\partial v}{\partial x_1}\frac{\Re(x^{N})}{|x|^{2N}}-\frac{\partial v}{\partial x_2}\frac{\Im(x^{N})}{|x|^{2N}}\bigg) \\&=
\frac12\frac{\partial }{\partial x_1}(|\nabla v|^2)\frac{\Re(x^{N})}{|x|^{2N}}-\frac12\frac{\partial }{\partial x_2}(|\nabla v|^2)\frac{\Im(x^{N})}{|x|^{2N}}-N|\nabla v|^2 \frac{\Re(x^{N+1})}{|x|^{2(N+1)}}
\\ &=\frac12\frac{\partial}{\partial x_1}\bigg(|\nabla v|^2 \frac{\Re(x^{N})}{|x|^{2N}}\bigg)-\frac12 \frac{\partial}{\partial x_2}\bigg(|\nabla v|^2 \frac{\Im(x^{N})}{|x|^{2N}}\bigg),
\end{aligned}$$
by which, applying again Gauss-Green formula,
$$\begin{aligned}&\int_{\Omega\setminus B_\e} \nabla v\cdot \nabla \bigg(\frac{\partial v}{\partial x_1}\frac{\Re(x^{N})}{|x|^{2N}}-\frac{\partial v}{\partial x_2}\frac{\Im(x^{N})}{|x|^{2N}}\bigg)
\\ & =\frac12 \int_{\partial\Omega\cup\partial B_\e} |\nabla v|^2 \bigg(\frac{\Re(x^{N})}{|x|^{2N}}\nu_1-\frac{\Im(x^{N})}{|x|^{2N}}\nu_2\bigg) d\sigma
.
\end{aligned}$$
Hence we deduce

\beq\label{csa}\begin{aligned}\int_{\Omega\setminus B_\e} (-\Delta v)&\bigg(\frac{\partial v}{\partial x_1}\frac{\Re(x^{N})}{|x|^{2N}}-\frac{\partial v}{\partial x_2}\frac{\Im(x^{N})}{|x|^{2N}}\bigg)dx\\ &=
\frac12 \int_{\partial\Omega\cup\partial B_\e} |\nabla v|^2 \bigg(\frac{\Re(x^{N})}{|x|^{2N}}\nu_1-\frac{\Im(x^{N})}{|x|^{2N}}\nu_2\bigg) d\sigma
\\ &\;\;\;\;-\int_{\partial \Omega\cup\partial B_\e }\frac{\partial v}{\partial \nu}\bigg(\frac{\partial v}{\partial x_1}\frac{\Re(x^{N})}{|x|^{2N}}-\frac{\partial v}{\partial x_2}\frac{\Im(x^{N})}{|x|^{2N}}\bigg)d\sigma .
\end{aligned}\eeq
Let us examine separately the boundary integrals over $\partial \Omega$ and over $\partial B_\e$:
taking into account of the homogeneous boundary condition we have that $\nabla v=\frac{\partial v}{\partial \nu}\,\nu$ on $\partial \Omega$, which implies $$\begin{aligned}&\frac12 \int_{\partial\Omega} |\nabla v|^2 \bigg(\frac{\Re(x^{N})}{|x|^{2N}}\nu_1-\frac{\Im(x^{N})}{|x|^{2N}}\nu_2\bigg) d\sigma-\\ &\;\;\;\;\int_{\partial \Omega}\frac{\partial v}{\partial \nu}\bigg(\frac{\partial v}{\partial x_1}\frac{\Re(x^{N})}{|x|^{2N}}-\frac{\partial v}{\partial x_2}\frac{\Im(x^{N})}{|x|^{2N}}\bigg)d\sigma
\\ &=-\frac12\int_{\partial \Omega}\bigg|\frac{\partial v}{\partial \nu}\bigg|^2 \bigg(\frac{\Re(x^{N})}{|x|^{2N}}\nu_1-\frac{\Im(x^{N})}{|x|^{2N}}\nu_2\bigg) d\sigma.
\end{aligned}$$
On the other hand,  $\nu=-\frac{x}{|x|}$ on $\partial B_\e$; consequently
 $$\begin{aligned}&\frac12 \int_{\partial B_\e} |\nabla v|^2 \bigg(\frac{\Re(x^{N})}{|x|^{2N}}\nu_1-\frac{\Im(x^{N})}{|x|^{2N}}\nu_2\bigg) d\sigma\\ &\;\;\;\;-\int_{\partial B_\e}\frac{\partial v}{\partial \nu}\bigg(\frac{\partial v}{\partial x_1}\frac{\Re(x^{N})}{|x|^{2N}}-\frac{\partial v}{\partial x_2}\frac{\Im(x^{N})}{|x|^{2N}}\bigg)d\sigma
 \\ &=-\frac{1}{\e^{2N+1}} \int_{\partial B_\e} \frac{|\nabla v|^2}{2} \bigg(\Re(x^{N})x_1-\Im(x^{N})x_2\bigg) \\ &\;\;\;\;+\frac{1}{\e^{2N+1}} \int_{\partial B_\e} \bigg(\frac{\partial v}{\partial x_1}x_1+\frac{\partial v}{\partial x_2}x_2\bigg)\bigg(\frac{\partial v}{\partial x_1}\Re(x^{N})-\frac{\partial v}{\partial x_2}\Im(x^{N})\bigg)d\sigma
 \\ &
=\frac{1}{2\e^{2N-1}}\int_{\partial B_\e}\bigg(\bigg|\frac{\partial v}{\partial x_1}\bigg|^2-\bigg|\frac{\partial v}{\partial x_2}\bigg|^2\bigg)\Re(x^{N-1})d\sigma\\ &\;\;\;\;-
\frac{1}{\e^{2N-1}}\int_{\partial B_\e}\frac{\partial v}{\partial x_1}\frac{\partial v}{\partial x_2}\Im(x^{N-1})d\sigma
 .\end{aligned}$$
where in the second identity we have used    Lemma \ref{lemma0}. 
By inserting the above two boundary estimates into \eqref{csa} we get
\beq\label{csa1}\begin{aligned}\int_{\Omega\setminus B_\e} (-\Delta v)\bigg(&\frac{\partial v}{\partial x_1}\frac{\Re(x^{N})}{|x|^{2N}}-\frac{\partial v}{\partial x_2}\frac{\Im(x^{N})}{|x|^{2N}}\bigg)dx\\ &=
-\frac12 \int_{\partial\Omega} \bigg|\frac{\partial  v}{\partial \nu}\bigg|^2 \bigg(\frac{\Re(x^{N})}{|x|^{2N}}\nu_1-\frac{\Im(x^{N})}{|x|^{2N}}\nu_2\bigg) d\sigma \\ &\;\;\;\;
+\frac{1}{2\e^{2N-1}}\int_{\partial B_\e}\bigg(\bigg|\frac{\partial v}{\partial x_1}\bigg|^2-\bigg|\frac{\partial v}{\partial x_2}\bigg|^2\bigg)\Re(x^{N-1})d\sigma\\ &\;\;\;\;-
\frac{1}{\e^{2N-1}}\int_{\partial B_\e}\frac{\partial v}{\partial x_1}\frac{\partial v}{\partial x_2}\Im(x^{N-1})d\sigma
\end{aligned}\eeq

Now let us pass to examine the right hand side of \eqref{leftright}: 
by using again Gauss Green formula and Lemma \ref{lemma0}
\beq\label{sgre}\begin{aligned}\int_{\Omega\setminus B_\e} \la  V(x) &e^{v}\bigg(\frac{\partial v}{\partial x_1}\Re(x^{N})-\frac{\partial v}{\partial x_2}\Im(x^{N})\bigg)  dx
\\ &=\int_{\Omega\setminus B_\e} \la V(x) \bigg(\frac{\partial e^v}{\partial x_1}\Re(x^{N})-\frac{\partial e^v}{\partial x_2}\Im(x^{N})\bigg)dx\\ &=
-\int_{\Omega\setminus B_\e} \la e^{v(x)}\bigg(\frac{\partial V(x)}{\partial x_1}\Re(x^{N})-\frac{\partial V(x)}{\partial x_2}\Im(x^{N})\bigg) dx
\\ &\;\;\;\;-\int_{\Omega\setminus B_\e} \la e^{v(x)}V(x)\bigg(\frac{\partial (\Re(x^N))}{\partial x_1}-\frac{\partial (\Im(x^N))}{\partial x_2}\bigg) dx
\\ &
\;\;\;\;+\int_{\partial \Omega\cup\partial B_\e}\la  V(x)e^v \Big(\Re(x^{N})\nu_1-\Im(x^{N})\nu_2\Big)d\sigma
\\ &=
-\int_{\Omega} \la e^{v(x)}\bigg(\frac{\partial V(x)}{\partial x_1}\Re(x^{N})-\frac{\partial V(x)}{\partial x_2}\Im(x^{N})\bigg) dx
\\ &
\;\;\;\;+\int_{\partial \Omega}\la  V(x)e^v \Big(\Re(x^{N})\nu_1-\Im(x^{N})\nu_2\Big)d\sigma+o(1) 
\end{aligned}\eeq as $\e\to 0^+$, where we have used the homogeneous boundary condition $v=0$ on $\partial \Omega$.

Let us insert \eqref{csa1} and \eqref{sgre} into \eqref{leftright} and letting $\e\to 0$ and we get the thesis.

\end{proof}

\subsection{Uniform Behaviour of blowing up solutions}
In this section we derive the asymptotic behaviour of blow-up solutions $v_k$ which is a direct consequence of a uniform estimate provided by \cite{bart3}  for bubbling solution to the Liouville equation with no boundary condition; roughly speaking, the analysis reveals that their profiles  differs
from global solutions of a Liouville type equation only by a uniformly bounded term. \begin{prop}\label{lililili} Let $N\in\N$. Assume that $\Omega$ is a smooth and bounded planar domain such that $0\in\Omega$ and the potential $V\in C^1(\overline\Omega)$ is bounded from below away from zero. If $v_k\in C(\overline{\Omega})$ is a sequence of blow-up solutions  for the problem  \eqref{eq000} satisfying \eqref{unifbound}-\eqref{eqq000}-\eqref{eqq0000}, then 
along a subsequence
\beq\label{cuccu}\la|x|^2V(x)e^{v_k}\rightharpoonup 8\pi (N+1) \de_0\eeq
weakly in the measure sense, where $\de_0$ denotes the Dirac delta with pole at $0$. Moreover, by setting $\beta_k$ as the maximum point of $v_k$ in $\Omega$: $$\beta_k\to 0 \, : \;v_k(\beta_k)=\max_{\Omega} v_k(x)\to +\infty,$$
the following holds\footnote{We use the notation $\sim$ to denote quantities which in the limit $\la\to 0^+$ are of the same order. }: 
\beq\label{cucu} \frac1\la_k\sim e^{\frac{v_k(\beta_k)}{2}} \eeq and 
\beq\label{cucucu}v_k(x)= \log \frac{e^{v_k(\beta_k)}}{(1+ \frac{\la_k V(\beta_k)}{8\al^2 }e^{v_k(\beta_k)}|x^{N+1}-b_k|^2)^2}+O(1) \hbox{ in }\Omega\eeq where $b_k=\beta_k^{N+1}$.
\end{prop}
\begin{proof}
We need to transform the equation into \eqref{eq000} into a Liouville equation with no boundary condition we can use the presence of the  free parameter  under the transformation
$$\bar{v}_k(x)= v_k(x)+\log \la_k$$
so that the sequence $\bar v_k$ satisfies 
$$\left\{\begin{aligned} & -\Delta \bar v_k= V(x)|x|^{2N}e^{\bar v_k}&\hbox{in }&\Omega,\\ & \into |x|^{2(N+1)}V(x)e^{\bar v_k}dx \leq C\end{aligned}\right.$$
and $0$ is the only blowup point of $\bar v_k$.
Then, by applying the  uniform estimate provided in   Theorem 1.4 of  \cite{bart3} for a sequence of solutions having a single blow up point:
$$\bar v_k(x)= \log \frac{e^{\bar v_k(\beta_k)}}{(1+ \frac{V(\beta_k)}{8\al^2 }e^{\bar v_k(\beta_k)}|x^{N+1}-b_k|^2)^2}+O(1)\hbox{ uniformly in }\Omega.$$
and \eqref{cucucu} follows. By evaluating \eqref{cucucu} for $x\in\partial \Omega$ (where $v_k(x)=0$) we get \eqref{cucu}.

\end{proof}

\subsection*{The Case of the Ball: Necessary Conditions for Blowing Up}
In this section we focus on the case when $N=1$ and $\Omega$ is the unit ball $B_1$ and we establish  an integral estimate for bubbling situation at 0. 
\begin{prop}\label{necessary} Assume that the sequence $v_k$ satisfies \eqref{eq0}-\eqref{eqq0} . Then, if $V$ verifies \eqref{eqq1}, the following holds
$$\frac{\partial v_k}{\partial x_i}(0)=o(1), \quad \int_{B_1} \la_k e^{v_k(x)}\bigg(\frac{\partial V(x)}{\partial x_1}x_1-\frac{\partial V(x)}{\partial x_2}x_2\bigg) dx=o(1).$$
\end{prop}

\begin{proof}

Consider Propositions \ref{propo1} for $N=1$ and $\Omega=B_1$: taking into account that $\nu=(x_1, x_2)$ on $ \partial B_1$, we get that the homogeneous boundary condition implies $$\nabla v_k=	\frac{\partial v_k}{\partial \nu}\nu=\frac{\partial v_k}{\partial \nu}(x_1,x_2)\hbox{ on }\partial B_1.$$ On the other hand the regular part $H(x,y)$ of the Green's function takes the form $H(x,y)=\frac{1}{4\pi}\log(1+|x|^2|y|^2-2x\cdot y) $ for $ x,y\in B_1$ (where $``\cdot"$ denotes the scalar product in $\R^2$), which gives  $$\frac{\partial H}{\partial x_i}(0,y)=-\frac{y_i}{2\pi}.$$ We deduce that for every $k$  the Pohozaev identity in Proposition \ref{propo1} takes the forms:

 \beq\label{milleuno}\begin{aligned}\frac12\int_{\partial B_1}\bigg|\frac{\partial v_k}{\partial \nu}\bigg|^2  x_i dx-&2N\la_k\int_{B_1}y_i|y|^2V(y) e^{v_k} dy +\la_k\int_{\partial B_1}x_i |x|^{2}V(x)  dx
\\ &=4N\pi\frac{\partial v_k}{\partial x_i}(0) +\la_k\into |x|^{2} \frac{\partial V}{\partial x_i}(x) e^{v_k}dx.\end{aligned}\eeq

Next observe that for any fixed $k$, since $v$ is two-times differentiable by standard regularity theory, 
 $$\frac{1}{\e}\int_{\partial B_\e}\bigg|\frac{\partial v_k}{\partial x_i}\bigg|^2 d\sigma=2\pi\bigg|\frac{\partial v_k}{\partial x_i}(0)\bigg|^2+o(1)\hbox{ as }\e\to 0^+$$ 
so that the Pohozaev identity in Proposition \ref{propo2} gives

\beq\label{milledue}\begin{aligned}\int_{B_1} \la e^{v_k(x)}&\bigg(\frac{\partial V(x)}{\partial x_1}x_1-\frac{\partial V(x)}{\partial x_2}x_2\bigg) dx
-\la_k\int_{\partial B_1}  V(x)\big(x_1^2-x_2^2\big)d\sigma\\ &=\frac12\int_{\partial B_1}\bigg|\frac{\partial v_k}{\partial \nu}\bigg|^2 \big(x_1^2-x_2^2\big) d\sigma-\pi \bigg|\frac{\partial v_k}{\partial x_1}(0)\bigg|^2+\pi\bigg|\frac{\partial v_k}{\partial x_2}(0)\bigg|^2
.\end{aligned}\eeq

We need to estimate $\frac{\partial v_k}{\partial \nu}$ on $\partial B_1$: by Poisson representation formula we have

$$\frac{\partial v_k}{\partial \nu}(x)=\la_k\int_{B_1} \frac{\partial G}{\partial \nu_x}(x-y)|y|^2V(y)e^{v_k(y)}dy\quad x\in\partial B_1.$$

Since the Green's function for the unit ball takes  the form $G(x,y)=\frac{1}{2\pi}\log\frac{1}{|x-y|}+\frac{1}{4\pi}\log\big(1+|x|^2|y|^2-2x\cdot y\big)$ for $x,y\in B_1$, 
so we have  $$\frac{\partial G}{\partial \nu_x}(x-y)=\frac{1}{2\pi}\frac{(y-x)\cdot x}{|x-y|^2}+\frac{1}{2\pi}\frac{(x|y|^2-y)\cdot x}{1+|y|^2-2x\cdot y}\qquad \forall x\in\partial B_1$$
and so 
$$\frac{\partial v_k}{\partial \nu}(x)=\frac{\la_k}{2\pi}\int_{B_1} \bigg(\frac{(y-x)\cdot x}{|x-y|^2}+\frac{(x|y|^2-y)\cdot x}{1+|y|^2-2x\cdot y}\bigg)|y|^{2}V(y)e^{v_k(y)} dy\quad \forall x\in\partial B_1.$$
By using \eqref{cuccu} and \eqref{cucucu}, the above integral formula gives 
$$\frac{\partial v_k}{\partial \nu}(x)\to  -8\hbox{ unif. for } x\in\partial B_1$$
We have thus proved that 
\beq\label{seiuno}\begin{aligned}\int_{\partial B_1} \bigg|\frac{\partial v_k}{\partial \nu}\bigg|^2\big(x_1^2-x_2^2) d\sigma&=64  \int_{\partial B_1}(x_1^2-x_2^2) d\sigma+o(1)=o(1)
\end{aligned}
\eeq
since $\int_{\partial B_1}(x_1^2-x_2^2) d\sigma=0$. Similarly
\beq\label{seidue}\begin{aligned}\int_{\partial B_1} \bigg|\frac{\partial v_k}{\partial \nu}\bigg|^2x_i d\sigma&=64 \int_{\partial B_1}x_i d\sigma+o(1)=o(1).
\end{aligned}
\eeq
Moreover again by \eqref{cuccu}, using that $\nabla V(0)=0$ by \eqref{eqq1}, we have 
$$\la_k\int_{B_1}y_i|y|^2V(y) e^{v_k} dy =o(1), \quad \la_k\int_{B_1} |x|^{2} \frac{\partial V}{\partial x_i}(x) e^{v_k}dx=o(1), $$ and $$\la_k\int_{\partial B_1}  V(x)\big(x_1^2-x_2^2\big)d\sigma=o(1)$$
By inserting the above estimates into \eqref{milleuno}-\eqref{milledue} we deduce the thesis.
\end{proof}

\subsection*{Proof of  Theorem \ref{th2}}
In this section we exploit the integral estimates obtained in Proposition \ref{necessary} in order to obtain necessary conditions for  the presence of bubbling solutions at $0$  for problem \eqref{eq0}-\eqref{eqq0}. 
In the following we assume that  $v_k$ is a sequence satisfying \eqref{eq0}-\eqref{eqq0} and the potential $V$ verifies \eqref{eqq1}. Moreover, after suitably rotating the coordinate system, we may assume that in a small neighborhood of $0$ the following expansion holds:
$$V(x)=V(0)+ \frac{a_{11}x_1^2+a_{22}x_2^2}{2}+o(|x|^2) \hbox{ as }x\to 0,$$
where $a_{ii}= \frac{\partial^2 V}{\partial x_i^2}(0),$ so that \beq\label{exx}\frac{\partial V}{\partial x_1}(0)x_1-\frac{\partial V}{\partial x_2}(0)x_2= a_{11}x_1^2-a_{22}x_2^2 +o(|x|^2).\eeq

 Now assume, in addition, that the sequence $v_k$ has the non-simple blow-up property.  It has been shown in
\cite{wei-zhang-jems} that Laplacian of coefficient functions must be  zero, i.e. $\Delta V(0)=0$ or, equivalently, $a_{11}+a_{22}=0$. Then $$  \la_k\int_{B_1} e^{v_k(x)}\bigg(\frac{\partial V}{\partial x_1}x_1-\frac{\partial V}{\partial x_2}x_2\bigg) dx=(a_{11}+o(1))\la_k\int_{B_1} |x|^2e^{v_k(x)}dx\to 16\pi\frac{a_{11}}{V(0)}.$$
Combining this convergence with that proved in Proposition \ref{necessary} we deduce \beq\label{both}a_{11}=a_{22}=0.\eeq
Then Theorem \ref{th2} follows.

  \begin{rem}\label{fails} We point out that if we try to apply our technique to $N\geq 2$, then the analogous of Proposition \ref{necessary} would give
$$\la_k\int_{B_1}  e^{v_k(x)}\bigg(\frac{\partial V(x)}{\partial x_1}\Re(x^N)-\frac{\partial V(x)}{\partial x_2}\Im(x^N)\bigg) dx=o(1).$$
Assume that the sequence $v_k$ has the non-simple blow-up property, so that, according to 
\cite{wei-zhang-jems} we have $\Delta V(0)=0$ or, equivalently, $a_{11}+a_{22}=0$.
Then, using Lemma \ref{lemma0} the above estimates becomes
$$\la_ka_{11}\int_{B_1}|x|^2  e^{v_k(x)}\Re(x^{N-1})dx+ \la_k\int_{B_1}o(|x|^{N+1})e^{v_k} dx 
=o(1).$$
Observe that only if $N=1$ the above integrals can be estimated using  \eqref{cuccu}, whereas for $N\geq 2$ they cannot be handled by \eqref{cuccu}.  This will explains why our method based on the combination of two Pohozaev identities 
fails to provide the results of Theorem \ref{th2} for $N\geq 2$. 
\end{rem}

\end{document}